\documentclass[hidelinks, 10pt]{article}

\usepackage[UKenglish]{babel}
\usepackage[utf8]{inputenc}
\usepackage{erewhon}
\usepackage{amsmath}
\usepackage{amsthm}
\usepackage{amsfonts}
\usepackage{amssymb}
\usepackage{mathtools}
\usepackage{dsfont} 
\usepackage[outline]{contour}

\usepackage[dvipsnames]{xcolor}
\usepackage{hyperref}
\usepackage{fullpage}
\usepackage{microtype}
\usepackage{newpxmath}
\usepackage{booktabs}
\usepackage{enumitem} 
\usepackage{comment}
\usepackage{authblk} 
\usepackage{fix-cm}

\usepackage{quiver}
\usepackage{adjustbox}

\tikzset{node distance=2cm, auto}
\tikzcdset{row sep/normal=2.7em,column sep/normal=3.5em}

\allowdisplaybreaks

\setlist[description]{font=\normalfont\bfseries\:\!}

\renewcommand\labelenumi{(\roman{enumi})}
\renewcommand\theenumi\labelenumi

\hypersetup{
    colorlinks=true,
    linkcolor=gray,
    filecolor=gray,      
    urlcolor=gray,
    citecolor=gray
    }
\numberwithin{equation}{section}

\newtheorem{theorem}{Theorem}[section]
\newtheorem{proposition}[theorem]{Proposition}
\newtheorem{lemma}[theorem]{Lemma}
\newtheorem{corollary}[theorem]{Corollary}

\theoremstyle{definition}
\newtheorem{definition}[theorem]{Definition}
\newtheorem{example}[theorem]{Example}

\newtheorem{remark}[theorem]{Remark}

\newenvironment{content}[1][blank]{\noindent{\large\bfseries\scshape\MakeLowercase{#1}}.\thinspace\noindent}{\vskip .5cm}
\newenvironment{acknowledgements}{\begin{content}[Acknowledgements]}{\par\noindent\rule{\textwidth}{1pt}\end{content}}

\renewenvironment{abstract}{\begin{content}[Abstract]}{\end{content}}

\newcommand{\CategoryFont}[1]{\mathsf{\uppercase{#1}}}

\newcommand{\NTimes}[1]{#1\text{ times}}

\newcommand{\+}{\oplus}
\newcommand{\x}{\otimes}

\renewcommand{\.}{\cdot}

\newcommand{\1}{\mathds{1}}
\newcommand{\I}{\mathrm{I}}

\renewcommand{\=}{\colon\kern-1ex=}
\renewcommand{\epsilon}{\varepsilon}
\renewcommand{\o}{\circ}

\newcommand{\<}{\langle}
\renewcommand{\>}{\rangle}
\newcommand{\op}{\mathsf{op}}

\renewcommand{\*}{\ast}

\newcommand{\N}{\mathbb{N}}
\newcommand{\Z}{\mathbb{Z}}
\newcommand{\Q}{\mathbb{Q}}
\newcommand{\R}{\mathbb R}

\renewcommand{\ker}{{\mathsf{ker}\:}}
\newcommand{\id}{\mathsf{id}}

\newcommand{\X}{\mathbb{X}}

\newcommand{\E}{\mathscr{E}}
\newcommand{\End}{\CategoryFont{End}}

\newcommand{\cRing}{\CategoryFont{cRing}}
\newcommand{\Mod}{\CategoryFont{Mod}}

\newcommand{\cAlg}{\CategoryFont{cAlg}}
\newcommand{\Aff}{\CategoryFont{Aff}}

\newcommand{\Tngoid}{\CategoryFont{TOID}}
\newcommand{\sTngoid}{\mathsf{s}\CategoryFont{TOID}}
\newcommand{\cSolid}{\CategoryFont{cSOLD}}
\newcommand{\Solid}{\mathrm{Solid}}
\newcommand{\Weil}{\CategoryFont{Weil}_1}

\newcommand{\D}{\mathsf D}

\newcommand{\T}{\mathrm{T}}
\newcommand{\TT}{\mathbb{T}}

\newcommand{\Leung}{\mathrm{L}}

\renewcommand{\l}{\mathrm{l}}

\newcommand{\Sym}{\CategoryFont{Sym}}
\newcommand{\cMon}{\CategoryFont{cMon}}
\newcommand{\Spec}{\mathrm{Spec}}

\title{Representable tangent structures for affine schemes}
\author{\uppercase{Marcello Lanfranchi,\: Jean-Simon Pacaud Lemay}}
\affil{\normalsize\textit{Macquarie University, School of Mathematical and Physical Sciences}}
\date{}
\makeatother
\makeatletter
\begin{document}

\maketitle

\begin{abstract}\noindent
The category of affine schemes is a tangent category whose tangent bundle functor is induced by K\"ahler differentials, providing a direct link between algebraic geometry and tangent category theory. Moreover, this tangent bundle functor is represented by the ring of dual numbers. How special is this tangent structure? Are there any other (non-trivial) tangent structure on the category of affine schemes? In this paper, we characterize the representable tangent structures on the category of affine schemes. To this end, we introduce a useful tool, the notion of tangentoids, which are precisely the objects in a monoidal category that induce a tangent structure via tensoring. Furthermore, coexponentiable tangentoids induce tangent structures on the opposite category. As such, we first prove that tangentoids in the category of commutative unital algebras are equivalent to commutative associative solid non-unital algebras, that is, commutative associative non-unital algebras whose multiplication is an isomorphism. From there, we explain how representable tangent structures on affine schemes correspond to finitely generated projective commutative associative solid non-unital algebras. In particular, for affine schemes over a principal ideal domain, we show that there are precisely two representable tangent structures: the trivial one and the one given by K\"ahler differentials.
\end{abstract}

\begin{acknowledgements}
The authors would like to thank Geoff Vooys, Steve Lack, and Richard Garner for useful discussions. This material is based upon work supported by the AFOSR under award number FA9550-24-1-0008, and by the ARC under DECRA award number DE230100303. 
\end{acknowledgements}


\vspace{-1cm}

\tableofcontents

\section{Introduction}
\label{section:introduction}
A tangent category, as introduced by Rosick\'y in~\cite{rosicky:tangent-cats} and later revisited and generalized by Cockett and Cruttwell in~\cite{cockett:tangent-cats}, is a categorical framework for differential geometry. Briefly, a tangent category (Definition~\ref{definition:tangent-cat}) is a category equipped with an endofunctor $\T$, where for an object $M$ we intuitively think of $\T M$ as an ``abstract tangent bundle'' over $M$, which comes equipped with various natural transformations that capture fundamental properties of the classical tangent bundle over a smooth manifold. Tangent category theory is now a well-established field of interest, with a rich literature and has found connections to other areas beyond differential geometry, such as to algebraic geometry, algebra, operad theory, computer science, etc. 

\par An important example of a tangent category, especially in relation to algebraic geometry, is the category of affine schemes (over a fixed commutative and unital ring $R$), which comes equipped with a tangent structure whose tangent bundle functor sends an affine scheme $\Spec(A)$ (where $A$ is a commutative and unital $R$-algebra) to the symmetric $A$-algebra of the module of K\"ahler differentials of $A$, i.e., $\T(\Spec(A)) = \Spec(\Sym_A\Omega_A)$. The corresponding $R$-algebra is called the ``fibré tangente''~\cite[Definition 16.5.12.I]{grothendieck:elements} or tangent algebra~\cite[Section 2.6]{jubin:tangent-monad-and-foliations} of $A$. This tangent structure on affine schemes, first described in~\cite{cockett:tangent-cats} and later extensively analyzed in~\cite{cruttwell:algebraic-geometry}, captures some key geometric constructions of algebraic geometry~\cite{cruttwell:algebraic-geometry,cruttwell:connections-algebraic-geometry}. Moreover, this tangent bundle functor is in fact representable, that is, it is the left adjoint to the functor $\Spec(R[\epsilon])\times-$, where $R[\epsilon]$ is the ring of dual numbers over the base ring $R$. More generally, representable tangent structures on categories with finite products are equivalently characterized by a special class of objects called infinitesimal objects (Definition~\ref{definition:infinitesimal-object}), which are those exponentiable objects $D$ whose corresponding representable functor $\T = (-)^D$ gives a tangent structure~\cite[Proposition~5.7]{cockett:tangent-cats}. In particular, $\Spec(R[\epsilon])$ is an infinitesimal object in the category of affine schemes over $R$. 

\par It is natural to wonder how special this tangent structure for affine schemes is. In this paper, we classify representable tangent structures on the category of affine schemes, or equivalently infinitesimal objects in the category of affine schemes. To do so, we first introduce tangentoids (Definition~\ref{definition:tangentoid}), which are objects $D$ in a monoidal category equipped with a suitable structure such that the endofunctor $\T(-) = D \otimes -$ gives a tangent structure (Proposition~\ref{proposition:tangentoid-tangent-structure1}). Coexponentiable tangentoids (Definition~\ref{definition:coexponentiable-tangentoid}) induce tangent structures on the opposite category. Therefore, coexponentiable tangentoids on a category with finite coproducts correspond precisely to infinitesimal objects in the opposite category. As such, classifying infinitesimal objects in the category of affine schemes over $R$ is equivalent to classifying coexponentiable tangentoids in the category $\cAlg_R$ of commutative and unital $R$-algebras.

\par We first show that every tangentoid $D$ in $\cAlg_R$ is isomorphic to an $R$-algebra $R\ltimes M$ (where $\ltimes$ denotes the square-zero extension) for some $R$-module $M$ equipped with a commutative and associative multiplication $\alpha\colon M\x_RM\to M$ which is also an isomorphism (Theorem~\ref{theorem:characterization-tangentoids}). Non-unital algebras whose multiplication is an isomorphism are better known as solid algebras \cite{gutierrez2015solid}. So in other words, we show that there is an equivalence between tangentoids in $\cAlg_R$ and commutative non-unital associative solid $R$-algebras. We then use a theorem of Niefield's~\cite[Theorem~4.3]{niefield:cartesianness} to construct an equivalence between coexponentiable tangentoids in $\cAlg_R$ and commutative non-unital associative solid $R$-algebras whose underlying $R$-module is finitely generated and projective (Theorem~\ref{theorem:characterization-infinitesimal-objects}), giving us our desired characterization of representable tangent structures on affine schemes.

\par As a special case, when the base ring $R$ is a principal ideal domain (PID), the only coexponentiable tangentoids in $\cAlg_R$ are the base ring $R$ and $R[\epsilon]$ (Theorem~\ref{theorem:PID-characterization}). In other words, the only representable tangent structures on the category of affine schemes over a PID $R$ are the trivial tangent structure (Example~\ref{example:trivial}) and the aforementioned tangent structure related to the module of K\"ahler differentials (Corollary~\ref{corollary:PID-infinitesimal-are-R-D}). When the base ring $R$ fails to be a PID, more representable tangent structures are allowed as we show in Example~\ref{example:non-PID-example}. Furthermore, even when the base ring is a PID, our classification does not put constraints on non-representable tangent structures. In particular, in Example~\ref{example:Q-tangentoid-non-exponentiable} we show that the ring $\Q$ of rational numbers induces a novel tangent structure on the category $\cRing$ of commutative and unital rings, which however, does not dualize to affine schemes.


\section{Background on tangent category theory}
\label{section:tangent-cats}
In this section, we recall the necessary background on tangent category theory. Tangent categories were initially introduced by Rosick\'y in~\cite{rosicky:tangent-cats} and later on, generalized by Cockett and Cruttwell in~\cite{cockett:tangent-cats}. A tangent category is a categorical framework for differential geometry. In this regard, the objects of a tangent category should be interpreted as generalized geometric spaces and morphisms as smooth maps between these spaces. To support this intepretation, a tangent category comes equipped with an endofunctor $\T$, called the tangent bundle functor, which assigns to every generalized geometric space $M$ a new space $\T M$, interpeted as the tangent bundle of $M$, that is, the geometric space of all tangent vectors of $M$. A tangent category also has a list of natural transformations which make $\T$ into a proper tangent bundle functor. For an in-depth introduction to tangent categories, we refer the reader to~\cite{cockett:tangent-cats, rosicky:tangent-cats, garner:embedding-theorem-tangent-cats}.

\par We assume that the reader is comfortable with basic concepts and notions of category theory, like functors, adjunction, (co)limits, and exponential objects. Ordinary categories are denoted by symbols such as $\X,\X'$. In a category $\X$, objects will be denoted by capital letters like $A,B,E,M,N$ and Hom sets by $\X(A,B)$, with morphisms denoted by lowercase letters $f,g\colon A\to B$. In this paper, we adopt functional composition, that is, the composition of two morphisms $f\colon A\to B$ and $g\colon B\to C$ is denoted $g\o f$, instead of diagrammatic composition as is sometimes used in other tangent category papers such as in~\cite{cockett:tangent-cats}.

\subsection{Tangent categories}
\label{subsection:tangent-cats}
We first recall the definition of additive bundles, which are commutative monoids in slice categories. The definition of an additive bundle involves pullbacks; we denote the pullback of morphisms $f\colon A\to C$ and $g\colon B\to C$ by $A\times_CB$ with projections $\pi_1\colon A\times_CB\to A$ and $\pi_2\colon A\times_CB\to B$. Moreover, given two morphisms $h_A\colon X\to A$ and $h_B\colon X\to B$ such that $f\o h_A=g\o h_B$, the unique morphism induced by the universal property of the pullback $A\times_CB$ is denoted by $\<h_A,h_B\>\colon X\to A\times_CB$. When we consider two pullbacks $A\times_CB$ and $A'\times_{C'}B'$ of two pairs of morphisms $f\colon A\to C$ and $g\colon B\to C$, and $f'\colon A'\to C'$ and $g'\colon B'\to C'$, given three morphisms $h_A\colon A\to A'$, $h_B\colon B\to B'$, and $h_C\colon C\to C'$ satisfying $f'\o h_A=h_C\o f$ and $g'\o h_B=h_C\o g$, the unique morphism $\<h_A\o\pi_1,h_B\o\pi_2\>\colon A\times_CB\to A'\times_{C'}B'$ induced by the universal property of $A'\times_{C'}B'$ is denoted by $h_A\times_{h_C}h_B\colon A\times_CB\to A'\times_{C'}B'$, or simply by $h_A\times_Ch_B$ when $h_C$ is the identity on $C$.

\begin{definition}
\label{definition:additive-bundle}
In a category $\X$ an \textbf{additive bundle}~\cite[Section 2.1]{cockett:tangent-cats} consists of a a pair of objects $E$ and $B$ of $\X$, a morphism $q\colon E\to M$, called the \textbf{projection}, for which the $n$-fold pullback
\begin{equation*}
\begin{tikzcd}
{E_n} & E \\
E & B
\arrow["{\pi_n}", from=1-1, to=1-2]
\arrow["{\pi_1}"', from=1-1, to=2-1]
\arrow["\lrcorner"{anchor=center, pos=0.125}, draw=none, from=1-1, to=2-2]
\arrow["\dots"{marking, allow upside down}, shift right=2, draw=none, from=1-2, to=2-1]
\arrow["q", from=1-2, to=2-2]
\arrow["q"', from=2-1, to=2-2]
\end{tikzcd}
\end{equation*}
of $q$ along itself exists, a morphism $z_q\colon B\to E$, called the \textbf{zero morphism}, and a morphism $s_q\colon E_2\to E$, called the \textbf{sum morphism}, satisfying the following equations:
\begin{equation*}
\begin{tikzcd}
E & B \\
B
\arrow["q", from=1-1, to=1-2]
\arrow["{z_q}", from=2-1, to=1-1]
\arrow[equals, from=2-1, to=1-2]
\end{tikzcd}
\qquad\hfill
\begin{tikzcd}
{E_2} & E \\
E & B
\arrow["{s_q}", from=1-1, to=1-2]
\arrow["{\pi_1}"', from=1-1, to=2-1]
\arrow["q", from=1-2, to=2-2]
\arrow["q"', from=2-1, to=2-2]
\end{tikzcd}
\end{equation*}
\begin{equation*}
\begin{tikzcd}
{E_2} & E \\
{E_2}
\arrow["{s_q}", from=1-1, to=1-2]
\arrow["{\langle \pi_2, \pi_1 \rangle}", from=2-1, to=1-1]
\arrow["{s_q}"', from=2-1, to=1-2]
\end{tikzcd}\qquad\hfill
\begin{tikzcd}
{E_3} & {E_2} \\
{E_2} & E
\arrow["{s_q\times_B\id_E}", from=1-1, to=1-2]
\arrow["{\id_E\times_Bs_q}"', from=1-1, to=2-1]
\arrow["{s_q}", from=1-2, to=2-2]
\arrow["{s_q}"', from=2-1, to=2-2]
\end{tikzcd}\qquad\hfill
\begin{tikzcd}
{E_2} & E \\
E
\arrow["{s_q}", from=1-1, to=1-2]
\arrow["{\<z_q\o q,\id_E\>}", from=2-1, to=1-1]
\arrow[equals, from=2-1, to=1-2]
\end{tikzcd}
\end{equation*}
As a shorthand, we denote additive bundles by triples $(q\colon M \to E, z_q, s_q)$, or simply as $(q, z_q, s_q)$, where $q\colon E \to M$ is the projection, $z_q$ is the zero morphism, and $s_q$ is the sum morphism. Given an endofunctor $\T\colon\X\to\X$, a \textbf{$\T$-additive bundle} is an additive bundle $(q, z, s)$ such that all $n$-fold pullbacks of the projection $q$ along itself are preserved by all iterates $\T^n$ of $\T$.
\end{definition}

If $(q, z, s)$ is an additive bundle, then the projection $q\colon E \to M$ is a commutative monoid in the slice category over $M$ with $s_q$ being the monoid multiplication and $z_q$ is the monoid unit. Thus, asking for pullbacks of $q$ along itself gives us products in the slice category, then the first two diagrams say $z_q$ and $s_q$ are maps in the slice category (with the first one also saying $z_q$ is a section of $q$), while the remaining three diagrams are the commutativity, associativity, and unit axioms respectively for a commutative monoid. There is also a natural notion of additive bundle morphism, which we do not review here and will instead write out the diagrams explicitly when necessary. When an additive bundle is also an Abelian group object in the slice category, we call this an Abelian group bundle, following the terminology used in tangent category theory literature (Abelian group bundles are also sometimes known as Beck modules~\cite{beck:triples-algebras-cohomology}). 

\begin{definition}
\label{definition:group-bundle}
An \textbf{Abelian group bundle}~\cite[Section 1]{rosicky:tangent-cats} is an additive bundle $(q\colon E \to M, z_q, s_q)$ equipped with an extra morphism $n_q\colon E\to E$, called \textbf{negation}, satisfying the following equation:
\begin{equation*}
\begin{tikzcd}
E & {E_2} \\
B & E
\arrow["{\<\id_E,n\>}", from=1-1, to=1-2]
\arrow["q"', from=1-1, to=2-1]
\arrow["{s_q}", from=1-2, to=2-2]
\arrow["{z_q}"', from=2-1, to=2-2]
\end{tikzcd}
\end{equation*}
As a shorthand, we denote Abelian group bundles as quadruples $(q\colon E \to M, z_q, s_q, n_q)$, or simply as $(q, z_q, s_q, n_q)$, where $n_q$ denotes the negation. For an endofunctor $\T\colon\X\to\X$, a \textbf{$\T$-Abelian group bundle} is an Abelian group bundle $(q, z_q, s_q, n_q)$ whose underlying additive bundle $(q, z_q, s_q)$ is a $T$-additive bundle. 
\end{definition}

We now turn our attention to reviewing the definition of a tangent category. The main difference between Rosick\'y's definition in~\cite{rosicky:tangent-cats} and Cockett and Cruttwell definition in~\cite{cockett:tangent-cats} is that the former assumes an Abelian group bundle structure while the latter is slightly more general since it only assumes an additive bundle structure. As such, Rosick\'y's notion of a tangent structure comes with an extra natural transformation capturing negatives and thus, following the terminology used in~\cite{cruttwell:algebraic-geometry}, we call this a Rosick\'y tangent category/structure (previously called a tangent category/structure with negatives). 

\begin{definition}
\label{definition:tangent-cat}
A \textbf{tangent structure}~\cite[Definition 2.3]{cockett:tangent-cats} on a category $\X$ is a tuple $\TT = (\T, p, z, s, l, c)$ consisting of the following:
\begin{description}
\item[Tangent bundle functor] An endofunctor $\T\colon\X\to\X$;

\item[Projection] A natural transformation $p_M\colon\T M\to M$, for which the $n$-fold pullback $\T_n$ of $p_M\colon\T M\to M$ along itself exists;

\item[Zero morphism] A natural transformation $z_M\colon M\to\T M$;

\item[Sum morphism] A natural transformation $s_M\colon\T_2M\to\T M$;

\item[Vertical lift] A natural transformation $l_M\colon\T M\to\T^2M$;

\item[Canonical flip] A natural transformation $c_M\colon\T^2M\to\T^2M$;
\end{description}
subject to the following conditions:
\begin{enumerate}
\item $(p_M,z_M,s_M)$ is a $\T$-additive bundle;

\item $(z_M,l_M)\colon(p_M,z_M,s_M)\to(\T p_M,\T z_M,\T s_M)$ is an additive bundle morphism, that is, the diagrams 
\begin{equation*}
\begin{tikzcd}
{\T M} & {\T^2M} \\
M & {\T M}
\arrow["l_M", from=1-1, to=1-2]
\arrow["p_M"', from=1-1, to=2-1]
\arrow["{\T p_M}", from=1-2, to=2-2]
\arrow["z_M"', from=2-1, to=2-2]
\end{tikzcd}\qquad\hfill
\begin{tikzcd}
{\T M} & {\T^2M} \\
M & {\T M}
\arrow["l_M", from=1-1, to=1-2]
\arrow["z_M", from=2-1, to=1-1]
\arrow["z_M"', from=2-1, to=2-2]
\arrow["{\T z_M}"', from=2-2, to=1-2]
\end{tikzcd}\qquad\hfill
\begin{tikzcd}[column sep=huge]
{\T_2M} & {\T\T_2M} \\
{\T M} & {\T^2M}
\arrow["{l_M\times_{z_M}l_M}", from=1-1, to=1-2]
\arrow["s_M"', from=1-1, to=2-1]
\arrow["{\T s_M}", from=1-2, to=2-2]
\arrow["l_M"', from=2-1, to=2-2]
\end{tikzcd}
\end{equation*}
commute;

\item $(\id_M,c_M)\colon(p_M,z_M,s_M)\to(\T p_M,\T z_M,\T s_M)$ is an additive bundle morphism, that is, the diagrams
\begin{equation*}
\begin{tikzcd}
{\T^2M} & {\T^2M} \\
{\T M} & {\T M}
\arrow["c_M", from=1-1, to=1-2]
\arrow["{p_{\T M}}"', from=1-1, to=2-1]
\arrow["{\T p_M}", from=1-2, to=2-2]
\arrow[equals, from=2-1, to=2-2]
\end{tikzcd}\qquad\hfill
\begin{tikzcd}
{\T^2M} & {\T^2M} \\
{\T M} & {\T M}
\arrow["c_M", from=1-1, to=1-2]
\arrow["{z_{\T M}}", from=2-1, to=1-1]
\arrow[equals, from=2-1, to=2-2]
\arrow["{\T z_M}"', from=2-2, to=1-2]
\end{tikzcd}\qquad\hfill
\begin{tikzcd}
{\T_2\T M} && {\T\T_2M} \\
{\T^2M} && {\T^2M}
\arrow["{c_M\times_{\T M}c_M}", from=1-1, to=1-3]
\arrow["{s_{\T M}}"', from=1-1, to=2-1]
\arrow["{\T s_M}", from=1-3, to=2-3]
\arrow["c_M"', from=2-1, to=2-3]
\end{tikzcd}
\end{equation*}
commute;

\item The following diagrams commute:
\begin{equation*}
\begin{tikzcd}
{\T M} & {\T^2M} \\
{\T^2M} & {\T^3M}
\arrow["l_M", from=1-1, to=1-2]
\arrow["l_M"', from=1-1, to=2-1]
\arrow["{\T l_M}", from=1-2, to=2-2]
\arrow["{l_{\T M}}"', from=2-1, to=2-2]
\end{tikzcd}\qquad\hfill
\begin{tikzcd}
{\T^2M} & {\T^2M} \\
& {\T^2M}
\arrow["c_M", from=1-1, to=1-2]
\arrow[equals, from=1-1, to=2-2]
\arrow["c_M", from=1-2, to=2-2]
\end{tikzcd}\qquad\hfill
\begin{tikzcd}
{\T^3M} & {\T^3M} & {\T^3M} \\
{\T^3M} & {\T^3M} & {\T^3M}
\arrow["{\T c_M}", from=1-1, to=1-2]
\arrow["{c_{\T M}}"', from=1-1, to=2-1]
\arrow["{c_{\T M}}", from=1-2, to=1-3]
\arrow["{\T c_M}", from=1-3, to=2-3]
\arrow["{\T c_M}"', from=2-1, to=2-2]
\arrow["{c_{\T M}}"', from=2-2, to=2-3]
\end{tikzcd}
\end{equation*}

\item The following diagrams commute:
\begin{equation*}
\begin{tikzcd}
{\T M} & {\T^2M} \\
& {\T^2M}
\arrow["l_M", from=1-1, to=1-2]
\arrow["l_M"', from=1-1, to=2-2]
\arrow["c_M", from=1-2, to=2-2]
\end{tikzcd}\qquad\hfill
\begin{tikzcd}
{\T^2M} & {\T^3M} & {\T^3M} \\
{\T^2M} && {\T^3M}
\arrow["{l_{\T M}}", from=1-1, to=1-2]
\arrow["c_M"', from=1-1, to=2-1]
\arrow["{\T c_M}", from=1-2, to=1-3]
\arrow["{c_{\T M}}", from=1-3, to=2-3]
\arrow["{\T l_M}"', from=2-1, to=2-3]
\end{tikzcd}
\end{equation*}

\item The vertical lift is universal in the sense that the following diagram
\begin{equation*}
\begin{tikzcd}
{\T_2M} & {\T\T_2M} & {\T^2M} \\
{\T M} \\
M && {\T M}
\arrow["{l_M\times_{z_M}z_{\T M}}", from=1-1, to=1-2]
\arrow["{\pi_1}"', from=1-1, to=2-1]
\arrow["{\T s_M}", from=1-2, to=1-3]
\arrow["{\T p_M}", from=1-3, to=3-3]
\arrow["p_M"', from=2-1, to=3-1]
\arrow["z_M"', from=3-1, to=3-3]
\end{tikzcd}
\end{equation*}
is a pullback diagram.
\end{enumerate}
A \textbf{Rosick\'y tangent structure}~\cite[Definition 3.3]{cockett:tangent-cats} is a tangent structure $\TT$ equipped with:
\begin{description}
\item[Negation] A natural transformation $n_M\colon\T M\to\T M$, such that $(p_M,z_M,s_M, n_M)$ is a $\T$-Abelian group bundle. 
\end{description}
A (\textbf{Rosick\'y}) \textbf{tangent category} consists of a category $\X$ equipped with a (Rosick\'y) tangent structure $\TT$.
\end{definition}

\par We shall denote the tangent bundle functor of a tangent structure with the same symbol adopted for the tangent structure, e.g., if the tangent structure is denoted by $\TT$, the corresponding tangent bundle functor is denoted by $\T$. If the symbol denoting the tangent structure is decorated with a subscript or with a superscript, the tangent bundle functor and each of the structural natural transformations are decorated in the same way, e.g., for a tangent structure $\TT'_\o$, the tangent bundle functor is denoted by $\T'_\o$, the projection by $p'_\o$ and so on.

\par Let us briefly provide some intuition for the definition of a (Rosický) tangent category. Tangent categories formalize the properties of the tangent bundle on smooth manifolds from classical differential geometry. An object $M$ in a tangent category can be interpreted as a generalized geometric space, and $\T M$ as its abstract tangent bundle, which represents the geometric space of the tangent vectors of $M$. The projection $p_M$ is the analogue of the natural projection from the tangent bundle to its base space which sends every tangent vector to its base point, making $\T M$ an abstract fibre bundle over $M$. The sum $s_M$ and the zero map $z_M$ make $\T M$ into a generalized version of a smooth vector bundle where each fibre is a commutative monoid. In a Rosický tangent category, each fibre is also an Abelian group. To explain the vertical lift, recall that in differential geometry, the double tangent bundle, that is, the tangent bundle of the tangent bundle, admits a canonical sub-bundle called the vertical bundle, which is isomorphic to the pullback $\T_2M$ of the projection along itself. The vertical lift $l_M$ is an analogue of the embedding of the tangent bundle into the double tangent bundle via the vertical bundle. The universal property of vertical lift is essential to generalize important properties of the tangent bundle from differential geometry, see~\cite[Section 2.5]{cockett:tangent-cats} for more details. Finally, the canonical flip $c_M$ is an analogue of the smooth involution of the same name on the double tangent bundle, which encodes the symmetry of the Hessian matrix.

\par As expected, the canonical example of a (Rosický) tangent category is the category of smooth manifolds, where the tangent bundle functor, as the name suggests, sends a smooth manifold to its tangent bundle. Here are now some other well-known examples of tangent categories that are of most interest for the story of this paper. Lists of other examples of tangent categories can be found in~\cite[Example 2.2]{cockett:differential-bundles} and~\cite[Example 2]{garner:embedding-theorem-tangent-cats}. 

\begin{example}
\label{example:trivial}
Every category $\X$ comes equipped with a Rosick\'y tangent structure, called the \textit{trivial tangent structure}, whose tangent bundle functor is the identity functor, and whose structural natural transformations are simply identities. 
\end{example}

\begin{example}
\label{example:biproducts} 
Every category $\X$ which admits finite biproducts, which we shall denote as $\+$, admits a tangent structure $\TT^\+$, whose tangent bundle functor $\T^\+$ sends an object $M\in\X$ to $M\+ M$. See~\cite[Lemma~3.1.1]{ikonicoff:operadic-algebras-tagent-cats} for full details.
\end{example}

\begin{example}
\label{example:cAlg-canonical}
Let $R$ be a commutative ring.\footnote{In this paper, rings and algebras are all assumed to be unital whether not specified otherwise.} The category $\cAlg_R$ of commutative $R$-algebras comes equipped with a Rosick\'y tangent structure $\TT^\epsilon$ whose tangent bundle functor sends an $R$-algebra $A$ to its ring of dual numbers:
\begin{align*}
&\T^\epsilon A \= A[\epsilon]
\end{align*}
where recall that $A[\epsilon] \= A[x]/(x^2)$. For full details of this example, see~\cite[Section 3]{cruttwell:algebraic-geometry}. 
\end{example}

\begin{example}
\label{example:affine-canonical}
Let $R$ be a commutative ring. The category $\Aff_R$ of affine schemes over $R$, regarded as the opposite of the category $\cAlg_R$, comes equipped with a Rosick\'y tangent structure $\TT^\Omega$ whose tangent bundle functor sends an algebra $A$ to the symmetric $A$-algebra of the $A$-module of K\"ahler differentials of $A$:
\begin{align*}
&\T^\Omega A\=\Sym_A\Omega_{A/R}
\end{align*}
 In~\cite[Definition 16.5.12.I]{grothendieck:elements}, Grothendieck calls $\T^\Omega A$ the ``fibré tangente'' (French for tangent bundle) of $A$, while in~\cite[Section 2.6]{jubin:tangent-monad-and-foliations}, Jubin calls $\T^\Omega A$ the tangent algebra of $A$. For full details of this example, see~\cite[Section 4]{cruttwell:algebraic-geometry}. 
\end{example}

\subsection{Representable tangent categories and infinitesimal objects}
\label{subsection:representable-tangent-cats}

We now briefly review representable tangent structures which, as the name suggests, is when the tangent bundle functor is representable. To do so, we must ask our tangent categories to have finite products and that the tangent bundle functor preserves them. So for a category $\X$ with finite products, we shall denote the product by $\times$ and the (chosen) terminal object by $\ast$.

\begin{definition}
\label{definition:cartesian-tangent-category}
A (\textbf{Rosick\'y}) \textbf{Cartesian tangent category}~\cite[Definition 2.8]{cockett:tangent-cats} consists of a (Rosick\'y) tangent category $(\X,\TT)$ with finite products, preserved by the tangent bundle functor, so in particular $\T(A \times B) \cong \T A \times \T B$ and $\T \ast \cong \ast$.
\end{definition} 

Recall that in a category $\X$ with finite products, an \textbf{exponentiable object} is an object $D$ for which the functor $D \times -$ admits a left adjoint, often denoted by $(-)^D$. Moreover, an endofunctor $F\colon \X \to \X$ is said to be \textbf{representable} if $F \cong (-)^D$, for some exponentiable object $D$. In this case, we say that $F$ is represented by $D$, or that $D$ represents the functor $F$.

\begin{definition}
\label{definition:representable-tangent-cats}
A \textbf{representable tangent structure}~\cite[Section~5.7]{cockett:tangent-cats} on a category $\X$ with finite products is a tangent structure $\TT$ such that for every positive integer $n$, the functor $\T_n\colon\X\to\X$ is a representable functor. A \textbf{representable tangent category} is a Cartesian tangent category whose tangent structure is representable. 
\end{definition}

By spelling out the definition, a representable tangent structure turns out to consist of a tangent structure on a category $\X$ with finite products together with a sequence $\{D_n\}_{n\in\N}$ of exponentiable objects of $\X$ such that, for each $n$, $D_n$ represents the endofunctor $\T_n$. By convention, we shall write $D\=D_1$ for the exponentiable object representing the tangent bundle functor $\T$.

\par Cockett and Cruttwell showed in~\cite[Proposition~5.7]{cockett:tangent-cats} that a representable tangent structure $\TT$ is fully characterized by a suitable structure on the object $D$ which represents the tangent bundle functor. An object with such a structure is called an infinitesimal object. We do not recall the full definition of an infinitesimal object here, as we will encounter the dual notion in the next section. For detailed introductions to infinitesimal objects and representable tangent categories, we invite the reader to see~\cite{cockett:tangent-cats,rosicky:tangent-cats,garner:embedding-theorem-tangent-cats}. 

\begin{definition}
\label{definition:infinitesimal-object}
An \textbf{infinitesimal object}~\cite[Definition~5.6]{cockett:tangent-cats} in a category $\X$ with finite products consists of:
\begin{description}
\item[Object] An object $D$ of $\X$;

\item[Coprojection] A morphism $p\colon\*\to D$, for which the $n$-fold pushout $D_n$ of $p$ along itself exists and is preserved by all iterates $D^n\times-$ of the functor $D\times-$;

\item[Cosum morphism] A morphism $s\colon D\to D_2$;

\item[Vertical colift] A morphism $l\colon D\times D\to D$;
\end{description}
such that the axioms in~\cite[Definition~5.6]{cockett:tangent-cats} hold, in particular $D^n$ and $D_n$ are exponentiable objects. A \textbf{Rosick\'y infinitesimal object} is an infinitesimal object $D$ equipped with:
\begin{description}
\item[Negation] A morphism $n\colon D\Rightarrow D$, satisfying the necessary axioms from~\cite[Section 4]{rosicky:tangent-cats}. 
\end{description}
As a shorthand, we will denote (Rosick\'y) infinitesimal objects simply by the underlying object $D$. 
\end{definition}

\begin{remark}
\label{remark:preservation-pushouts}
In the definition of a (Rosick\'y) infinitesimal object, the requirement that the $n$-fold pushout of the coprojection along itself is preserved by each functor $D^n\times-$, is implied by $D$ being exponentiable. We decided to highlight this axiom since in the next section, we consider objects which satisfy the same (of the dual of the) axioms of an infinitesimal object, but for which the (co)exponentiability axiom is dropped. 
\end{remark}

\begin{proposition}~\cite[Proposition~5.7]{cockett:tangent-cats}
\label{proposition:infinitesimal-objects-representable-tang-cats}
For a category $\X$ with finite products, if $\TT$ is a (Rosick\'y) representable tangent structure over $\X$ and $D$ represents the tangent bundle functor, then $D$ comes equipped with the structure of an (Rosick\'y) infinitesimal object of $\X$. Conversely, if $D$ is an infinitesimal object of $\X$, $D$ induces a representable tangent structure on $\X$ whose tangent bundle functor is $(-)^D$ and the structural natural transformations are obtained by exponentializing the corresponding structural morphisms of $D$.
\end{proposition}

Here are now examples of representable tangent categories and their infinitesimal objects. 

\begin{example}
\label{example:trivial-infinitesimal}
For any category with finite products, the trivial tangent structure of Example~\ref{example:trivial} is representable, where the corresponding infinitesimal object is the terminal object $\ast$.
\end{example}

\begin{example}
\label{example:dual-numbers-infinitesimal}
The tangent structure $\TT^\Omega$ of Example~\ref{example:affine-canonical} on $\Aff_R$ is representable~\cite[Section 5.2]{cockett:tangent-cats}, where the corresponding infinitesimal object is the ring of dual numbers $R[\epsilon]$. The goal of this paper is to characterize the representable tangent structures on the category $\Aff_R$ of affine schemes. 
\end{example}

Representability of a tangent structure is a special case of a more general phenomenon called adjunctability. In particular, when each functor $\T_n$ of a tangent structure admits left adjoint $\T_n^\o$, the opposite category admits an ``adjoint'' tangent structure $\TT^\o$~\cite[Proposition 5.17]{cockett:tangent-cats}. It is possible to simplify the assumptions for adjunctability of a tangent structure when the base category has suitable pushouts~\cite[Corollary~2.2.4]{ikonicoff:operadic-algebras-tagent-cats}. Since representability plays a crucial role in the story of this paper, we want to prove a similar simplification in the representable case. For starters, recall the following technical lemma:

\begin{lemma}~\cite[Lemma~2.2.3]{ikonicoff:operadic-algebras-tagent-cats}
\label{lemma:Frankland-lemma}
Let $\T\colon\X\to\X$ be an endofunctor of a category $\X$ and $p_M\colon\T M\to M$ be a natural transformation for which the $n$-fold pullback $\T_nM$ of $p_M$ along itself exists in $\X$. Suppose also that the functor $\T$ admits a left adjoint $\T^\o$ and let $\eta_M\colon M\to\T\T^\o M$ and $\varepsilon_M\colon\T^\o\T M\to M$ denote the unit and the counit of the adjunction, respectively. Consider the mate of $p_M$ along this adjunction, that is, the following composite:
\begin{align*}
&p^\o_M\colon M\xrightarrow{\eta_M}\T\T^\o M\xrightarrow{p_{\T^\o M}}\T^\o M
\end{align*}
If the $n$-fold pushout $\T_n^\o M$ of $p^\o_M$ exists in $\X$, then the functor $\T_n^\o$ is a left adjoint of $\T_n$.
\end{lemma}

The next lemma simplifies the definition of infinitesimal objects when the base category admits finite pullbacks.

\begin{lemma}
\label{lemma:Frankland-lemma-representability}
In a category $\X$ with finite pullbacks and a terminal object (so in particular, $\X$ has finite products), let $D$ be an object equipped with a morphism $p\colon\*\to D$ such that the $n$-fold pushout $D_n$ along itself exists. If $D$ is exponentiable then so is each $D_n$.
\begin{proof}
The exponentiability condition on $D$ means that the functor $\T\=D\times-\colon\X\to\X$ admits a right adjoint. Moreover, we also have a natural transformation $p_A\colon A\to\T A$ defined as $p_A\=p\times\id_M\colon A \to D \times A$. In order to apply Lemma~\ref{lemma:Frankland-lemma}, let us consider $D$ in the opposite category $\X^\op$ of $\X$. Thus, the functor $\T\colon\X^\op\to\X^\op$ admits a left adjoint $\T^\o$ and we have a natural transformation $p_A\colon\T A\to A$. Since the $n$-fold pushout of $p$ exists in $\X$, the $n$-fold pullback of $p$ exists in $\X^\op$. Moreover, since $\X$ has finite pullbacks, $\X^\op$ has finite pushouts, the mate $p^\o$ of $p$ along the adjunction $\T^\o\dashv\T$ admits the $n$-fold pushout $\T_n^\o$ along itself. Thus, by Lemma~\ref{lemma:Frankland-lemma}, the functor $\T_n=D_n\times-$ admits a left adjoint $\T_n^\o$ in $\X^\op$, making $D_n$ coexponentiable in $\X^\op$ and therefore exponentiable in $\X$.
\end{proof}
\end{lemma}


\section{Tangentoids}
\label{section:tangentoids}
Characterizing representable tangent structures is equivalent to characterizing infinitesimal objects. Moreover, thanks to~\cite[Proposition 5.17]{cockett:tangent-cats}, if $\mathbb{T}$ is a representable tangent structure with associated infinitesimal object $D$, then $D \times -$ induces a tangent structure on the opposite category, where $\times$ becomes a coproduct. Such a tangent structure can be generalized to arbitrary monoidal categories. As such, in this section, we introduce the notion of a tangentoid in a monoidal category, which consists of an object such that tensoring with said object gives a tangent structure. 

Tangentoids will be useful for the main objective of this paper since they provide an alternative characterization of infinitesimal objects. Indeed, in a setting with (co)products, an infinitesimal object is precisely a coexponentiable tangentoid in the opposite category. Therefore, characterizing representable tangent structures and infinitesimal objects is equivalent to characterizing coexponentiable tangentoids in the opposite category. So, in particular, since the category of affine schemes over $R$ is equivalent to the opposite category of $\cAlg_R$, coexponentiable tangentoids in $\cAlg_R$ are the same as infinitesimal objects in $\Aff_R$. 

\subsection{Definition of a tangentoid}
\label{subsection:tangentoids}

We assume the reader is comfortable with the basic notions of monoidal category theory, including the ones of monoidal functor, monoidal unit, and monoidal product. For monoidal categories, we shall use letters like $\E,\E'$ to denote the underlying category, while the monoidal product is denoted by $\x$ and the unit by $\I$. We will not explicitly write the associator and the unitors, and treat monoidal categories as if they were strict. Thus, monoidal categories will be denoted by triples $(\E,\x,\I)$, or simply by their underlying category $\E$ when there is no confusion.

\par As mentioned above, a tangentoid in a monoidal category is precisely an object $D$ for which the functor $D \otimes -$ gives a tangent structure. Thus, the axioms of a tangentoid are analogous to the axioms of a tangent structure. The name tangentoid will be justified in Section~\ref{subsection:tangentoids-monoids-structure}. 

\begin{definition}
\label{definition:tangentoid}
A \textbf{tangentoid} in a monoidal category $(\E,\x,\I)$ consists of:
\begin{description}
\item[Object] An object $D$ of $\E$;

\item[Projection] A morphism $p\colon D\to\I$, for which the $n$-fold pullback $D_n$ of $p$ along itself exists and is preserved by all iterates $D^n\x-$ of the functor $D\x-$;

\item[Zero morphism] A morphism $z\colon\I\to D$;

\item[Sum morphism] A morphism $s\colon D_2\to D$;

\item[Vertical lift] A morphism $l\colon D\to D\x D$;

\item[Canonical flip] A morphism $c\colon D\x D\to D\x D$;
\end{description}
subject to the following conditions:
\begin{enumerate}
\item $(p,z,s)$ is a $(D\x-)$-additive bundle;

\item $(z,l)\colon(D,p,z,s)\to(D\x D,\id_D\x p,\id_D\x z,\id_D\x s)$ is an additive bundle morphism, that is, the diagrams 
\begin{equation*}
\begin{tikzcd}
D & {D\x D} \\
\I & D
\arrow["l", from=1-1, to=1-2]
\arrow["p"', from=1-1, to=2-1]
\arrow["{\id_D\x p}", from=1-2, to=2-2]
\arrow["z"', from=2-1, to=2-2]
\end{tikzcd}\quad\hfill
\begin{tikzcd}
D & {D\x D} \\
\I & D
\arrow["l", from=1-1, to=1-2]
\arrow["z", from=2-1, to=1-1]
\arrow["z"', from=2-1, to=2-2]
\arrow["{\id_D\x z}"', from=2-2, to=1-2]
\end{tikzcd}\quad\hfill
\begin{tikzcd}[column sep=huge]
{D_2} & {D\x D_2} \\
D & {D\x D}
\arrow["{l\times_zl}", from=1-1, to=1-2]
\arrow["s"', from=1-1, to=2-1]
\arrow["{\id_D\x s}", from=1-2, to=2-2]
\arrow["l"', from=2-1, to=2-2]
\end{tikzcd}
\end{equation*}
commute;

\item $(\id_M,c)\colon(D\x D,p\x\id_D,z\x\id_D,s\x\id_D)\to(D\x D,\id_D\x p,\id_D\x z,\id_D\x s)$ is an additive bundle morphism, that is, the diagrams
\begin{equation*}
\begin{tikzcd}
{D\x D} & {D\x D} \\
D & D
\arrow["c", from=1-1, to=1-2]
\arrow["{p\x\id_D}"', from=1-1, to=2-1]
\arrow["{\id_D\x p}", from=1-2, to=2-2]
\arrow[equals, from=2-1, to=2-2]
\end{tikzcd}\quad\hfill
\begin{tikzcd}
{D\x D} & {D\x D} \\
D & D
\arrow["c", from=1-1, to=1-2]
\arrow["{z\x\id_D}", from=2-1, to=1-1]
\arrow[equals, from=2-1, to=2-2]
\arrow["{\id_D\x z}"', from=2-2, to=1-2]
\end{tikzcd}\quad\hfill
\begin{tikzcd}[column sep=large]
{D_2\x D} && {D\x D_2} \\
{D\x D} && {D\x D}
\arrow["{c\times_Dc}", from=1-1, to=1-3]
\arrow["{s\x\id_D}"', from=1-1, to=2-1]
\arrow["{\id_D\x s}", from=1-3, to=2-3]
\arrow["c"', from=2-1, to=2-3]
\end{tikzcd}
\end{equation*}
commute;

\item The following diagrams commute:
\begin{equation*}
\begin{tikzcd}
D & {D\x D} \\
{D\x D} & {D\x D\x D}
\arrow["l", from=1-1, to=1-2]
\arrow["l"', from=1-1, to=2-1]
\arrow["{\id_D\x l}", from=1-2, to=2-2]
\arrow["{l\x\id_D}"', from=2-1, to=2-2]
\end{tikzcd}\quad\hfill
\begin{tikzcd}
{D\x D} & {D\x D} \\
& {D\x D}
\arrow["c", from=1-1, to=1-2]
\arrow[equals, from=1-1, to=2-2]
\arrow["c", from=1-2, to=2-2]
\end{tikzcd}
\end{equation*}
\begin{equation*}
\begin{tikzcd}
{D\x D\x D} & {D\x D\x D} & {D\x D\x D} \\
{D\x D\x D} & {D\x D\x D} & {D\x D\x D}
\arrow["{\id_D\x c}", from=1-1, to=1-2]
\arrow["{c\x\id_D}"', from=1-1, to=2-1]
\arrow["{c\x\id_D}", from=1-2, to=1-3]
\arrow["{\id_D\x c}", from=1-3, to=2-3]
\arrow["{\id_D\x c}"', from=2-1, to=2-2]
\arrow["{c\x\id_D}"', from=2-2, to=2-3]
\end{tikzcd}
\end{equation*}

\item The following diagrams commute:
\begin{equation*}
\begin{tikzcd}
D & {D\x D} \\
& {D\x D}
\arrow["l", from=1-1, to=1-2]
\arrow["l"', from=1-1, to=2-2]
\arrow["c", from=1-2, to=2-2]
\end{tikzcd}\quad\hfill
\begin{tikzcd}
{D\x D} & {D\x D\x D} & {D\x D\x D} \\
{D\x D} && {D\x D\x D}
\arrow["{l\x\id_D}", from=1-1, to=1-2]
\arrow["c"', from=1-1, to=2-1]
\arrow["{\id_D\x c}", from=1-2, to=1-3]
\arrow["{c\x\id_D}", from=1-3, to=2-3]
\arrow["{\id_D\x l}"', from=2-1, to=2-3]
\end{tikzcd}
\end{equation*}

\item The vertical lift is universal, that is, the following diagram
\begin{equation*}
\begin{tikzcd}
{D_2} && {D\x D_2} & {D\x D} \\
D \\
\I &&& D
\arrow["{l\times_z(z\x\id_D)}", from=1-1, to=1-3]
\arrow["{\pi_1}"', from=1-1, to=2-1]
\arrow["{\id_D\x s}", from=1-3, to=1-4]
\arrow["{\id_D\x p}", from=1-4, to=3-4]
\arrow["p"', from=2-1, to=3-1]
\arrow["z"', from=3-1, to=3-4]
\end{tikzcd}
\end{equation*}
is a pullback diagram.
\end{enumerate}
A \textbf{Rosick\'y tangentoid} is a tangentoid $D$ equipped with:
\begin{description}
\item[Negation] A morphism $n\colon D\to D$ such that $(p,z,s,n )$ is a $(D\x-)$-Abelian group bundle. 
\end{description}
As a shorthand, we will denote (Rosick\'y) tangentoids simply by the underlying object $D$. 
\end{definition}

\begin{remark}
\label{remark:additivity-axiom}
In Definition~\ref{definition:tangentoid}, we require each $m$-fold pullback of the projection along itself to be preserved by each functor $D^n\x-$. Concretely, this means that the unique morphism induced by the universal property of the pullback
\begin{align*}
&D^n\x D_m\to\underbrace{(D^n\x D)\times_D{\dots}\times_D(D^n\x D)}_{\NTimes{m}}
\end{align*}
is an isomorphism. Unlike infinitesimal objects where preservation of the pushouts comes from exponentiability, since tangentoids are not assumed to be coexponentiable, we must ask for preservation of these pullbacks in our definition. 
\end{remark}

The universality of the vertical lift simplifies for Rosick\'y tangentoids.

\begin{lemma}
\label{lemma:rosicky-tangentoids}
If $D$ is a tangentoid in a monoidal category $\E$, the following diagram is a triple equalizer:
\begin{equation*}
\begin{tikzcd}
&& \I \\
D & {D\x D} && D
\arrow["z", bend left, from=1-3, to=2-4]
\arrow["l", from=2-1, to=2-2]
\arrow["{p\x p}", bend left, from=2-2, to=1-3]
\arrow["{\id_D\x p}"', bend right=50, from=2-2, to=2-4]
\arrow["{p\x\id_D}"{description}, from=2-2, to=2-4]
\end{tikzcd}
\end{equation*}
Moreover, for a Rosick\'y tangentoid in a monoidal category $\E$, assuming the other axioms hold, the universality of the vertical lift is equivalent to the above diagram being a triple equalizer.
\end{lemma}
\begin{proof}
The proof is similar to the one of~\cite[Lemma~3.10]{cockett:tangent-cats}.
\end{proof}

There is also a natural notion of morphisms of tangentoids, which is analogous to the notion of morphisms of tangent structures~\cite[Definition 2.27]{cockett:tangent-cats}. 

\begin{definition}
\label{definition:morphism-tangentoids}
A \textbf{morphism of tangentoids} in a monoidal category $(\E,\x,\I)$ from a (Rosick\'y) tangentoid $D$ to a (Rosick\'y) tangentoid $D'$ consists of a morphism $f\colon D\to D'$ of $\E$ which is compatible with the tangentoid structures in the sense that the following diagrams commute:
\begin{equation*}
\begin{tikzcd}
D & {D'} \\
& \I
\arrow["f", from=1-1, to=1-2]
\arrow["p"', from=1-1, to=2-2]
\arrow["{p'}", from=1-2, to=2-2]
\end{tikzcd}\quad\hfill
\begin{tikzcd}
D & {D'} \\
& \I
\arrow["f", from=1-1, to=1-2]
\arrow["z", from=2-2, to=1-1]
\arrow["{z'}"', from=2-2, to=1-2]
\end{tikzcd}\quad\hfill
\begin{tikzcd}
{D_2} & {D'_2} \\
D & {D'}
\arrow["{f\times_\I f}", from=1-1, to=1-2]
\arrow["s"', from=1-1, to=2-1]
\arrow["{s'}", from=1-2, to=2-2]
\arrow["f"', from=2-1, to=2-2]
\end{tikzcd}
\end{equation*}
\begin{equation*}
\begin{tikzcd}
D & {D'} \\
{D\x D} & {D'\x D'}
\arrow["f", from=1-1, to=1-2]
\arrow["l"', from=1-1, to=2-1]
\arrow["{l'}", from=1-2, to=2-2]
\arrow["{f\x f}"', from=2-1, to=2-2]
\end{tikzcd}\quad\hfill
\begin{tikzcd}
{D\x D} & {D'\x D'} \\
{D\x D} & {D'\x D'}
\arrow["{f\x f}", from=1-1, to=1-2]
\arrow["c"', from=1-1, to=2-1]
\arrow["{c'}", from=1-2, to=2-2]
\arrow["{f\x f}"', from=2-1, to=2-2]
\end{tikzcd}
\end{equation*}
\end{definition}

Tangentoids on a monoidal category $(\E,\x,\I)$ together with their morphisms form a category denoted by $\Tngoid(\E)$. Morphisms of Rosick\'y tangentoids are just morphisms of tangentoids since the compatibility with the negation is automatic. 

\begin{lemma}
\label{lemma:morphisms-rosicky-tangentoids}
If $f\colon D\to D'$ is a morphism of tangentoids and $D$ and $D'$ are Rosick\'y with negations $n$ and $n'$, respectively, the following diagram commutes:
\begin{equation*}
\begin{tikzcd}
D & {D'} \\
D & {D'}
\arrow["f", from=1-1, to=1-2]
\arrow["n"', from=1-1, to=2-1]
\arrow["{n'}", from=1-2, to=2-2]
\arrow["f"', from=2-1, to=2-2]
\end{tikzcd}
\end{equation*}
\end{lemma}
\begin{proof}
Using the compability of $f$ with the additive structure of $D$ and $D'$, we have:
\begin{align*}
&s'\o\<f,f\o n\>=s'\o(f\times_\I f)\o\<\id_D,n\>=f\o s\o\<\id_D,n\>=f\o z\o p=z'\o p'\o f=z'\o p'\o n'\o f
\end{align*}
where we used that $p'\o n'=p'$. Therefore, using the cancellation property, we first compute that:
\begin{align*}
&s'\o(\id_{D'}\times_\I s')\o\<n'\o f,\<f,f\o n\>\>=s'\o\<n'\o f,s'\o\<f,f\o n\>\>\\
&\quad=s'\o\<n'\o f,z'\o p'\o n'\o f\>=s'\o\<\id_{D'},z'\o p'\>\o n'\o f=n'\o f
\end{align*}
and we can also compute that:
\begin{align*}
&s'\o(\id_{D'}\times_\I s')\o\<n'\o f,\<f\o n,f\>\>=s'\o(s'\times_\I\id_{D'})\o\<\<n'\o f,f\>,f\o n\>\\
&\quad=s'\o\<s'\o\<n',\id_{D'}\>\o f,f\o n\>=\<s'\o\<z'\o p',\id_{D'}\>,\id_{D'}\>\o f\o n=f\o n
\end{align*}
So, we conclude that $n'\o f = f\o n$ as desired. 
\end{proof}

Examples of tangentoids can be found below in Section~\ref{subsection:symmetric-tangentoids}.

\subsection{Tangent structures parametrized by tangentoids}
\label{subsection:leung-functor}
In~\cite{leung:weil-algebras}, Leung suggested a different approach to tangent category theory. Leung first defined the category $\Weil$, which is the symmetric monoidal category generated by the family of rigs $W_n\=\N[x_1\,x_n]/(x_ix_j)$ and by the following structure morphisms:
\begin{description}
\item[Projection] $p\colon W=\N[x]/(x^2)\to\N$, which sends $x$ to $0$;

\item[Zero morphism] $z\colon\N\to W$, which sends $1$ to itself;

\item[Sum morphism] $s\colon W_2\to W$, which sends $x_1$ and $x_2$ to $x$;

\item[Vertical lift] $l\colon W\to W\x W$, which sends $x$ to $x\x x$;

\item[Canonical flip] $c\colon W\x W\to W\x W$, which coincides with the symmetry.
\end{description}
Leung then showed that a tangent structure on a category $\X$ is equivalent to a suitable strong monoidal functor $\Leung[\TT]\colon\Weil\to\End(\X)$ from the symmetric monoidal category $\Weil$ to the monoidal category of endofunctors of $\X$, which preserves a suitable class of limits. Concretely, the strong monoidal functor $\Leung[\TT]$ sends $W$ to the tangent bundle functor $\T$ and each structural morphism to the corresponding structural natural transformation of $\TT$. Leung also introduced the notion of a tangent structure \textit{internal} to a monoidal category $(\E,\x,\I)$ as a strong monoidal functor $\Leung[\TT]\colon\Weil\to(\E,\x,\I)$ which preserves suitable limits~\cite[Definition~14.2]{leung:weil-algebras}. By unwrapping Leung's definition, one can show that this definition is precisely equivalent to our notion of a tangentoid. 

\begin{proposition}
\label{proposition:leung-functor}
An internal tangent structure in a monoidal category $(\E,\x,\I)$ is equivalent to a tangentoid in $(\E,\x,\I)$.
\end{proposition}
\begin{proof}
Each internal tangent structure in $(\E,\x,\I)$ consists of a strong monoidal functor $\Leung[D]\colon\Weil\to(\E,\x,\I)$ that preserves some limit diagrams. In particular, $\Leung[D]$ sends $W$ to an object $D$ of $\E$, the projection $p\colon W\to\N$ to a morphism $p\colon D\to\I$, the zero morphism $z\colon\N\to W$ to a morphism $z\colon\I\to D$, the sum morphism $s\colon W_2\to W$, to a morphism $s\colon D_2\to D$, where $D_2$ must be the pullback of $p$ along itself, the vertical lift $l\colon W\to W\x W$ to a morphism $l\colon D\to D\x D$, and the canonical flip $c\colon W\x W\to W\x W$ to a morphism $c\colon D\x D\to D\x D$. However, since $W\x-$ induces a tangent structure on $\Weil$, $(D,p,z,s,l,c)$ satisfies the axioms of a tangentoid in $\E$. Conversely, each strong monoidal functor $\Leung[D]\colon\Weil\to(\E,\x,\I)$ that preserves the required limits is fully determined by the assignment of $W$ and the structural morphisms of $\Weil$.
\end{proof}

From the equivalence between internal tangent structures and tangentoids, one can show that a tangentoid on $\E$ is equivalent to giving a tangent structure on $\E$ whose tangent bundle functor is the functor $D\x-$ for some object $D$. 

\begin{proposition}
\label{proposition:tangentoid-tangent-structure1}
If $D$ is a (Rosick\'y) tangentoid in a monoidal category $(\E,\x,\I)$, the functor $D\x-$ together with the natural transformations $p\x-$, $z\x-$, $s\x-$, $l\x-$, and $c\x-$ (and $n\x-$) defines a (Rosick\'y) tangent structure on $\E$, denoted by $[D\x-]$.
\end{proposition}
\begin{proof}
From Leung's approach, to give a tangent structure on $\E$ is equivalent to give a strong monoidal functor $\Leung[\TT]\colon\Weil\to\End(\E)$ that preserves a suitable class of limits. By Proposition~\ref{proposition:leung-functor}, a tangentoid is equivalent to a strong monoidal functor $\Leung[D]\colon\Weil\to\E$ that preserves a suitable class of limits. By postcomposing $\Leung[D]$ by the functor which sends an object $A$ of $\E$ to an endofunctor $A\x-\colon\E\to\E$, one obtains a strong monoidal functor $\Weil\to\E\to\End(\E)$, which defines a tangent structure on $\E$.
\end{proof}

The converse of Proposition~\ref{proposition:tangentoid-tangent-structure1} holds as well by evaluating the tangent bundle functor at the monoidal unit. 

\begin{proposition}
\label{proposition:tangentoid-tangent-structure2}
If $\TT$ is a (Rosick\'y) tangent structure on $\E$ whose tangent bundle functor $\T$ is of type $D\x-$ for some object $D$, then $D$ comes equipped with the structure of a (Rosick\'y) tangentoid, whose structural natural transformations correspond to $p_\I$, $z_\I$, $s_\I$, $l_\I$, and $c_\I$ (and $n_\I$).
\end{proposition}
\begin{proof}
Consider a tangent structure $\TT$ on $\E$ whose tangent bundle functor $\T$ is the functor $D\x-$ for some object $D$. Via the Leung correspondence, this is equivalent to a suitable strong monoidal functor $\Leung[\TT]\colon\Weil\to\End(\E)$. The evaluation functor $\mathrm{ev}_\I$, which sends an endofunctor $F\colon E\to E$ to $F\I$ lifts to a strong monoidal functor $\End(\E)\to\E$. By postcomposing by $\Leung[\TT]$, we define an internal tangent structure $\Weil\to\End(\E)\to\E$, which, by Proposition~\ref{proposition:leung-functor}, is equivalent to a tangentoid on $\E$.
\end{proof}

Propositions~\ref{proposition:tangentoid-tangent-structure1} and~\ref{proposition:tangentoid-tangent-structure2} show an equivalence between tangentoids and tangent structures whose tangent bundle functor is of type $\T=D\x-$, for some object $D$.

\begin{theorem}
\label{theorem:tangentoid-tangent-structure}
Tangentoids on a monoidal category $\E$ are equivalent to those tangent structures on $\E$ whose tangent bundle functor $\T$ is a functor of type $D\x-$, for some object $D$.
\end{theorem}

From Theorem~\ref{theorem:tangentoid-tangent-structure}, one can immediately see the connection between tangent categories and tangentoids, as illustrated in the following example.

\begin{example}
\label{example:tangentoid-tangent-cats}
In Leung's approach, a tangent structure on a category $\X$ is equivalent to a strong monoidal functor $\Leung[\TT]\colon\Weil\to\End(\X)$ preserving suitable limits. Thus, by Theorem~\ref{theorem:tangentoid-tangent-structure}, a tangent structure on $\X$ is precisely a tangentoid in the monoidal category $\End(\X)$ of endofunctors of $\X$ (where the monoidal product is given by composition and the monoidal unit is the identity functor). Thus, a (Rosick\'y) tangent category is equivalently a category $\X$ with a (Rosick\'y) tangentoid in $\End(\X)$. 
\end{example}

\subsection{Symmetric tangentoids}
\label{subsection:symmetric-tangentoids}

We now consider tangentoids in \emph{symmetric} monoidal categories, specifically those whose canonical flip coincides with the symmetry. For a symmetric monoidal category, the symmetry will be denoted by $\sigma$ and symmetric monoidal categories will be denoted by quadruples $(\E,\x,\I,\sigma)$. 

\begin{definition}
\label{definition:symmetric-tangentoid}
In a symmetric monoidal category $(\E,\x,\I,\sigma)$, a \textbf{symmetric tangentoid} is a tangentoid whose canonical flip $c\colon D\x D\to D\x D$ coincides with the symmetry $\sigma_{D,D}\colon D\x D\to D\x D$.
\end{definition}

The full subcategory of symmetric tangentoids of a symmetric monoidal category $(\E,\x,\I,\sigma)$ is denoted by $\sTngoid(\E)$. The axioms of symmetric tangentoids involving the canonical flip, but the cocommutativity of the vertical lift, are in fact a consequence of the axioms of the symmetry. Thus, the definition of a symmetric tangentoid can be greatly simplified. 

\begin{proposition}
\label{proposition:simplify-symmetric-tangentoids}
In the context of a symmetric monoidal category $(\E,\x,\I,\sigma)$, the axioms of a symmetric (Rosick\'y) tangentoid involving the canonical flip, excluding the axiom $c\o l=l$, holds for free.
\end{proposition}
\begin{proof}
Let us consider a symmetric tangentoid in a symmetric monoidal category whose symmetry is denoted by $\sigma$. As we discussed in our convention, we do not explicitly write the associator and the unitors of the monoidal category. Let us start by considering the compatibilities between the projection and the zero morphism, which correspond to the commutativity of the following two diagrams:
\begin{equation*}
\begin{tikzcd}
{D\x D} && {D\x D} \\
& D
\arrow["\sigma", from=1-1, to=1-3]
\arrow["{\id_D\x p}"', from=1-1, to=2-2]
\arrow["{p\x\id_D}", from=1-3, to=2-2]
\end{tikzcd}\quad\hfill
\begin{tikzcd}
{D\x D} && {D\x D} \\
& D
\arrow["\sigma", from=1-1, to=1-3]
\arrow["{\id_D\x z}", from=2-2, to=1-1]
\arrow["{z\x\id_D}"', from=2-2, to=1-3]
\end{tikzcd}
\end{equation*}
Both diagrams commute by the naturality of $\sigma$ and by the compatibility between the symmetry and the left and right unitors.
\par Let us prove the additivity. Consider the following diagrams
\begin{equation*}
\begin{tikzcd}
{D\x D_2} &&& {D_2\x D} \\
& {(D\x D)\times_D(D\x D)} & {(D\x D)\times_D(D\x D)} \\
{D\x D} &&& {D\x D}
\arrow["\sigma", from=1-1, to=1-4]
\arrow["\cong"', from=1-1, to=2-2]
\arrow["{\id_D\x\pi_k}"', from=1-1, to=3-1]
\arrow["\cong", from=1-4, to=2-3]
\arrow["{\pi_k\x\id_D}", from=1-4, to=3-4]
\arrow["{\sigma\times_D\sigma}", from=2-2, to=2-3]
\arrow["{\pi_k}"', from=2-2, to=3-1]
\arrow["{\pi_k}", from=2-3, to=3-4]
\arrow["\sigma"', from=3-1, to=3-4]
\end{tikzcd}
\end{equation*}
parametrized by $k=1,2$. The outer square diagram commutes by the naturality of $\sigma$, the two triangles commute by the definition of the isomorphism maps, and the bottom trapeze commutes by the definition of $\sigma\times_D\sigma$. Employing the universal property of the pullback, we conclude that the top trapeze commutes as well. Therefore, $c\=\sigma$ is additive. 

Involutivity of $c$ is equivalent to the symmetry of $\sigma$. To prove the pentagon and the hexagon axioms, let us start by observing that the hexagon axiom of the symmetry makes the following diagram commute:
\begin{equation}
\label{equation:diagram:hexagon-symmetry}
\begin{tikzcd}
{D\x D\x D} && {D\x D\x D} \\
& {D\x D\x D}
\arrow["{\sigma_{D\x D,D}}", from=1-1, to=1-3]
\arrow["{\id_D\x\sigma}"', from=1-1, to=2-2]
\arrow["{\id_\D\x\sigma}"', from=2-2, to=1-3]
\end{tikzcd}
\end{equation}
Therefore
\begin{equation*}
\begin{tikzcd}
{D\x D} && {D\x D} \\
{D\x D\x D} && {D\x(D\x D)} \\
& {D\x(D\x D)}
\arrow["\sigma", from=1-1, to=1-3]
\arrow["{l\x\id_D}"', from=1-1, to=2-1]
\arrow["{\id_D\x l}", from=1-3, to=2-3]
\arrow["{\sigma_{D\x D,D}}", from=2-1, to=2-3]
\arrow["{\id_\D\x\sigma}"', from=2-1, to=3-2]
\arrow["{\sigma\x\id_D}"', from=3-2, to=2-3]
\end{tikzcd}
\end{equation*}
commutes, since the top square commutes by the naturality of $\sigma$ and the bottom triangle is precisely Diagram~\eqref{equation:diagram:hexagon-symmetry}. Finally, the hexagon axiom reads as follows:
\begin{equation*}
\begin{tikzcd}
{D\x D\x D} & {D\x D\x D} & {D\x D\x D} \\
\\
{D\x D\x D} & {D\x D\x D} & {D\x D\x D}
\arrow["{\sigma\x\id_D}", from=1-1, to=1-2]
\arrow["{\id_D\x\sigma}"', from=1-1, to=3-1]
\arrow["{\sigma_{D\x D,D}}"{description}, from=1-1, to=3-2]
\arrow["{\id_D\x\sigma}", from=1-2, to=1-3]
\arrow["{\sigma_{D\x D,D}}"{description}, from=1-2, to=3-3]
\arrow["{\sigma\x\id_D}", from=1-3, to=3-3]
\arrow["{\id_D\x\sigma}"', from=3-1, to=3-2]
\arrow["{\sigma\x\id_D}"', from=3-2, to=3-3]
\end{tikzcd}
\end{equation*}
The left and the right triangles are precisely Diagram~\eqref{equation:diagram:hexagon-symmetry}, while the central parallelogram is the naturality of $\sigma_{D\x D,D}$.
\end{proof}

Here are now examples of symmetric tangentoids, which in particular recapture tangent structures we have already encountered. 

\begin{example}
\label{example:unit} For any monoidal category $(\E,\x,\I)$, the monoidal unit $I$ is a Rosick\'y tangentoid whose tangent structure is trivial. Since for every object we have $I \otimes A = A$, it follows that the induced tangent structure from $I$ is precisely the trivial tangent structure given in Example~\ref{example:trivial}. For a symmetric monoidal category $(\E,\x,\I, \sigma)$, $I$ is also a symmetric Rosick\'y tangentoid. 
\end{example}

\begin{example}
\label{example:symmetric-tangentoid-vector-spaces} Let $\Mod_R$ be the category of modules over a commutative ring $R$, which is a symmetric monoidal category given by the usual algebraic tensor product of modules $\otimes$. In $(\Mod_R,\x,R, \sigma)$, the $R$-module $D\=R\+R$ is a symmetric Rosick\'y tangentoid. The projection $p\colon R\+R \to R$ is given by the first projection, $p(x,y) =x$, with pullbacks given by:
\begin{align*}
&D_n \= R\+ \underbrace{R \+ \hdots \+ R}_{n\text{ times}}
\end{align*}
The zero $z\colon R\+R \to R$ is given by the first injection, $z(x) = (x,0)$, the sum $s\colon R \+ R \+ R \to R \+ R$ sums the second and third components, $s(x,y,z) = (x, y+z)$, while the negation $n\colon R\+R \to R \+ R$ makes the second component negative, $n(x,y) = (x,-y)$. Lastly, the vertical lift $l \colon R \+ R \to (R \+ R) \x (R \+ R)$ inserts $l(x,y) = (x,0) \x (1,0) + (0,1) \x (0,y)$. Moreover, observe that for a $R$-module $M$, we have that $D \x M \cong M \+ M$. As such the induced tangent structure from $D$ is equivalent to the biproduct tangent structure given in Example~\ref{example:biproducts}. In fact, it is easy to see that this example generalizes to any symmetric monoidal category with finite biproducts that are preserved by the monoidal product. 
\end{example}

\begin{example}
\label{example:tangentoid-dual-numbers} Let $R$ be a commutative ring. Then $\cAlg_R$ is a symmetric monoidal category where the monoidal product is given by the usual tensor product of $R$-algebras $\x_R$. The ring of dual numbers $R[\epsilon]$ is a symmetric Rosick\'y tangentoid in $\cAlg_R$. The projection $p\colon R[\epsilon]\to R$ sends $\epsilon$ to $0$, that is, the augmentation map, with pullbacks given by $R[\epsilon]_n \= R[\epsilon_1,\hdots, \epsilon_n] \= R[x_1, \hdots, x_n]/(x_ix_j)$. The zero morphism $z\colon R\to R[\epsilon]$ sends $1$ to itself, that is, the canonical inclusion, the sum morphism $R[\epsilon_1,\epsilon_2]\to R[\epsilon]$ sends both $\epsilon_1$ and $\epsilon_2$ to $\epsilon$, while the negation $n\colon R[\epsilon]\to R[\epsilon]$ sends $\epsilon$ to $-\epsilon$. The lift $l\colon R[\epsilon]\to R[\epsilon] \otimes_R [\epsilon']$ sends $\epsilon$ to $\epsilon \otimes \epsilon$. Moreover, observe that for a commutative $R$-algebra $A$, we have that $A[\epsilon] \cong R[\epsilon] \otimes_R A$. Thus, the induced tangent structure by $R[\epsilon]$ is equivalent to the tangent structure given in Example~\ref{example:cAlg-canonical}. 
\end{example}

\subsection{Tangentoids in (co)Cartesian monoidal categories}
\label{subsection:tangentoids-cocartesian}

Our primary goal is to characterize tangentoids in the category $\cAlg_R$ of commutative algebras. However, the monoidal structure of $\cAlg_R$ is coCartesian, that is, the monoidal product is given by binary coproducts and the unit is initial. In fact, recall that any category with finite coproducts can be seen as a symmetric monoidal category, which we refer to as a coCartesian monoidal category (in other words, a symmetric monoidal category whose monoidal product is a coproduct and whose monoidal unit is an initial object). Therefore, we may consider tangentoids in coCartesian monoidal categories. In such a setting, every tangentoid is in fact symmetric. 

\begin{proposition}
\label{proposition:symmetric-tangentoid-cocartesian}
In a coCartesian monoidal category, every (Rosick\'y) tangentoid is symmetric.
\end{proposition}
\begin{proof} 
Let $D$ be a tangentoid in a coCartesian monoidal category. Since the monoidal unit $\I$ is an initial object, the zero morphism of $D$ must be the unique morphism $!\colon\I\to D$. Therefore, the compatibility between the canonical flip and the zero morphism reads as follows:
\begin{equation*}
\begin{tikzcd}
{D\x D} & {D\x D} \\
D
\arrow["c", from=1-1, to=1-2]
\arrow["{\iota_1}", from=2-1, to=1-1]
\arrow["{\iota_2}"', from=2-1, to=1-2]
\end{tikzcd}
\end{equation*}
where $\iota_j\colon D \to D \x D$ are the injection maps of the coproduct. From involutivity, we also get that $c \o \iota_2=\iota_1$. However, symmetry isomorphism $\sigma$ also satisfies $\sigma \o \iota_1 = \iota_2$ and $\sigma \o \iota_2 = \iota_1$. Then by the universal property of coproducts, it follows that $c = \sigma$. Thus, every tangentoid is symmetric.
\end{proof}

\begin{corollary}
\label{corollary:tangentoid-in-cALG-are-symmetric}
Every tangentoid in the coCartesian monoidal category $\cAlg_R$ is symmetric.
\end{corollary}

It is natural to wonder if a similar simplification occurs for tangentoids in a category with finite products. It turns out that for Cartesian monoidal categories (which are symmetric monoidal categories whose monoidal product is a product and the unit is a terminal object), the only tangentoids are terminal objects. 

\begin{proposition}
\label{proposition:symmetric-tangentoid-cartesian}
The only tangentoid (up to unique isomorphism) in a Cartesian monoidal category is the terminal object.
\end{proposition}
\begin{proof}
Let $D$ be a tangentoid in a Cartesian monoidal category. We need to show that $D$ is isomorphic to the monoidal unit $I$, which is a terminal object. By the universal property of the terminal object, the projection morphism $p\colon D \to I$ must be the unique morphism from $D$ to $I$. By Lemma~\ref{lemma:rosicky-tangentoids}, the following diagram is a triple equalizer
\begin{equation*}
\begin{tikzcd}
D & {D\times D} && D
\arrow["l", from=1-1, to=1-2]
\arrow["{\id_D\times p}"', curve={height=18pt}, from=1-2, to=1-4]
\arrow["{p\o z\times p}", curve={height=-18pt}, from=1-2, to=1-4]
\arrow["{p\times\id_D}"{description}, from=1-2, to=1-4]
\end{tikzcd}
\end{equation*}
where we used the strict notation for the unitors, that is, $D\times\*=D=\*\times D$. However, $\id_D\times p\colon D\times D\to D$ and $p\times\id_D\colon D\times D\to D$ are respectively the projections $\pi_1$ and $\pi_2$ of the product. Therefore, by the universal property of the product, $l$ must coincide with the morphism $\Delta\o z\o p$, where $\Delta\=\<\id_D,\id_D\>\colon D\to D\times D$ is the diagonal map. Now consider a morphism $f\colon D\to D$. From the universal property of the terminal object, we get that $p\o f=p$, and thus:
\begin{align*}
&l\o f=\Delta\o z\o p\o f=\Delta\o z\o p=l
\end{align*}
Therefore, by the universal property of the vertical lift, it follows that there is a unique comparison morphism $D\to D$ between $l\o f$ and $l$, which must be the identity. Since this works for every morphism $f$, we get that $z\o p=\id_D$. However, since $p \o z = \id_I$, it follows that $D \cong I$, and so $D$ is also a terminal object. So we conclude that every tangentoid in a Cartesian monoidal category is a terminal object.
\end{proof}

\begin{corollary}
\label{corollary:tangentoid-AFF}
The only tangentoid (up to unique isomorphism) in the Cartesian monoidal category $\Aff_R$ of affine schemes over $R$ is the terminal object $R$.
\end{corollary}

\subsection{Coexponentiable tangentoids}
\label{subsection:coexp-tangentoids}
By Theorem~\ref{theorem:tangentoid-tangent-structure}, a tangentoid in a monoidal category $\E$ is equivalent to a tangent structure on $\E$ whose tangent bundle functor is a functor $D\x-$, for some object $D$. It follows that $D_n\x-$ corresponds to $\T_n$. When each of the functor $D_n\x-$ admits a left adjoint, by the adjoint theorem~\cite[Proposition~5.17]{cockett:tangent-cats}, the opposite category of $\E$ admits a tangent structure. This section introduces coexponentiable tangentoids, which are precisely those tangentoids whose associated tangent structure $D\x-$ can be ``dualized'' to define a new tangent structure on the opposite category.

\begin{definition}
\label{definition:coexponentiable-tangentoid}
A tangentoid $D$ in a monoidal category $(\E,\x,\I)$ is \textbf{coexponentiable} provided that each functor $D_n\x-\colon\E\to\E$ admits a left adjoint. 
\end{definition}

\begin{lemma}
\label{lemma:coexponentiable-tangentoid-tangent-category}
Given a monoidal category $(\E,\x,\I)$, every coexponentiable tangentoid $D$ in $\E$ induces a tangent structure on the opposite category $\E^\op$ whose tangent bundle functor is the (opposite of the) left adjoint $\T^\o$ of $D\x-$ in $\E$.
\end{lemma}

We may now precisely state the relationship between tangentoids and infinitesimal objects.

\begin{lemma}
\label{lemma:infinitesimal-objects-tangentoids}
In a category $\X$ with finite products, an infinitesimal object is equivalent to a coexponentiable tangentoid in the opposite category $\X^\op$ (seen as a coCartesian monoidal category). 
\end{lemma}

In other words, an infinitesimal object is equivalently an ``\textit{exponentiable cotangentoid}'' in a Cartesian monoidal category. Moreover, recall that by Proposition~\ref{proposition:symmetric-tangentoid-cocartesian}, every tangentoid in a coCartesian monoidal category is symmetric. Thus, every infinitesimal object induces a symmetric tangentoid in the opposite category, which explains why the extra canonical flip structural map is not necessary in the definition of an infinitesimal object. 

\par Thanks to Lemma~\ref{lemma:infinitesimal-objects-tangentoids}, coexponentiable tangentoids in the coCartesian monoidal category $\cAlg_R$ correspond to infinitesimal objects in the category $\Aff_R$. Dually, infinitesimal objects of $\cAlg_R$ are precisely coexponentiable tangentoids in $\Aff_R$, which we know are trivial. 

\begin{corollary}
\label{corollary:no-infinitesimal-objects-CALG}
The only infinitesimal object in the coCartesian monoidal category $\cAlg_R$ is (up to unique isomorphism) the initial object $R$. 
\end{corollary}

\subsection{Monoid structure of a tangentoid}
\label{subsection:tangentoids-monoids-structure}

To conclude this section, we explain how every tangentoid has a canonical monoid structure, justifying the name tagentoid which comes from the contration of \textit{tangent} and \textit{monoid}. This monoid structure in fact comes from the canonical monad structure on the tangent bundle functor. Moreover, in symmetric monoidal categories, the monoid structure on a tangentoid is always commutative, regardless of whether the tangentoid is symmetric or not. 

Recall that in a monoidal category $(\E,\x,\I)$, a \textbf{monoid} is a triple $(A, \mu, \eta)$ consisting of an object $A$ equipped with a morphism $\mu\colon A\x A\to A$ and a morphism $\eta\colon\I\to A$, such that the following diagrams commute:
\begin{equation*}
\begin{tikzcd}
{A\x A\x A} & {A\x A} \\
{A\x A} & A
\arrow["{\id_A\x\mu}", from=1-1, to=1-2]
\arrow["{\mu\x\id_A}"', from=1-1, to=2-1]
\arrow["\mu", from=1-2, to=2-2]
\arrow["\mu"', from=2-1, to=2-2]
\end{tikzcd} \begin{tikzcd}
A & {A\x A} & A \\
& A
\arrow["{\id_A\x\eta}", from=1-1, to=1-2]
\arrow[equals, from=1-1, to=2-2]
\arrow["\mu"{description}, from=1-2, to=2-2]
\arrow["{\eta\x\id_A}"', from=1-3, to=1-2]
\arrow[equals, from=1-3, to=2-2]
\end{tikzcd}
\end{equation*}
In a symmetric monoidal category $(\E,\x,\I, \sigma)$, a monoid $(A, \mu, \eta)$ is \textbf{commutative} if the following diagram commutes:
\begin{equation*}
\begin{tikzcd}
& {A\x A} \\
{A\x A} & A
\arrow["\mu", from=1-2, to=2-2]
\arrow["{\sigma_{A,A}}", from=2-1, to=1-2]
\arrow["\mu"', from=2-1, to=2-2]
\end{tikzcd}
\end{equation*}

\begin{lemma}
\label{lemma:tangentoids-are-monoids}
For a tangentoid $D$ in a monoidal category $(\E,\x,\I)$, the following morphisms
\begin{equation*}
\eta\colon\I\xrightarrow{z}D\qquad
\mu\colon D\x D\xrightarrow{\<\id_D\x p,p\x\id_D\>}D\times_\I D=D_2\xrightarrow{s}D
\end{equation*}
equip $D$ with a monoid structure. Moreover, when $(\E,\x,\I, \sigma)$ is a symmetric monoidal category, the monoid structure on $D$ is commutative.
\end{lemma}
\begin{proof}
Thanks to Theorem~\ref{theorem:tangentoid-tangent-structure}, $D$ equips $\E$ with a tangent structure $[D\x-]$, so in particular, the functor $\T\=D\x-$ is a tangent bundle functor. However, in every tangent category, the tangent bundle functor comes with a monad structure~\cite[Proposition~3.4]{cockett:tangent-cats}. Furthermore, a functor of type $D\x-$ is a monad if and only if $D$ is a monoid. By spelling out the monad structure on $\T$ one finds out that the unit of $D$ is $z$ and that the multiplication is $\mu$.
\par Now, suppose that $D$ is a (non-necessarily symmetric) tangentoid in a symmetric monoidal category $\E$. By naturality of the symmetry, we first obtain that:
\begin{equation*}
(\id_D\x p)\o\sigma=p\x\id_D\qquad(p\x\id_D)\o\sigma=\id_D\x p
\end{equation*}
Therefore, using these identities and the commutativity of $s$, we compute that:
\begin{gather*}
\quad\mu\o\sigma=s\o\<\id_D\x p,p\x\id_D\>\o\sigma= s\o\<(\id_D\x p)\o\sigma,(p\x\id_D)\o\sigma\>= s\o\<p\x\id_D,\id_D\x p\>\\
= s\o\tau\o\<\id_D\x p,p\x\id_D\>=s\o\<\id_D\x p,p\x\id_D\>=\mu
\end{gather*}
So we conclude that $D$ is also a commutative monoid. 
\end{proof}

Since every tangentoid is a monoid, it is natural to wonder if it is also a tangentoid in the category of monoids of the monoidal category. This holds when the base monoidal category is symmetric. So for a symmetric monoidal category $(\E,\x,\I, \sigma)$, we denote its category of commutative monoids by $\cMon(\E)$, which recall inherits the monoidal structure of $\E$. In particular, $\x$ becomes a coproduct for commutative monoids, and therefore $\cMon(\E)$ is a coCartesian monoidal category. 

\begin{lemma}
\label{lemma:tangentoids-in-monoids}
If $D$ is a symmetric tangentoid in a symmetric monoidal category $(\E,\x,\I, \sigma)$, the commutative monoid $(D,\mu,z)$ of Lemma~\ref{lemma:tangentoids-are-monoids} is a (symmetric) tangentoid in the coCartesian monoidal category $\cMon(\E)$. 
\end{lemma}
\begin{proof}
For any tangent structure, the tangent structure maps are all monad morphisms. Thus, it follows that the projection $p$, the zero $z$, the sum $s$, the multiplication $l$, and the canonical flip $c$ of $D$ are all in fact monoid morphisms. Furthermore, since the forgetful functor $\cMon(\E)\to\E$ is conservative, that is, reflects isomorphisms, it also reflects limits~\cite{kelly:conservative-functors}. Therefore, we conclude that $((D,\mu,z),p,z,s,l)$ is a (necessarily symmetric) tangentoid in $\cMon(\E)$.
\end{proof}

The operation which sends a symmetric tangentoid in a symmetric monoidal category $\E$ to a tangentoid in $\cMon(\E)$ extends to a functor
\begin{align*}
&M\colon\sTngoid(\E)\to\sTngoid(\cMon(\E))
\end{align*}
which sends a morphism $f\colon D\to D'$ of tangentoids to the morphism $f\colon(D,\mu,z)\to(D',\mu',z')$ of tangentoids of $\cMon(\E)$. In certain cases, this functor is in fact an isomorphism. 

\begin{proposition}
\label{proposition:symmetric-tangentoids}
When the forgetful functor $U\colon\sTngoid(\cMon(\E))\to\sTngoid(\E)$ which forgets the commutative monoid structure preserves limits, the functor $M\colon\sTngoid(\E)\to\sTngoid(\cMon(\E))$ is an isomorphism of categories whose inverse is $U$.
\end{proposition}
\begin{proof}
Firstly, since, by assumption, the functor $U\colon\sTngoid(\cMon(\E))\to\sTngoid(\E)$ preserves limits, $U$ sends tangentoids to tangentoids. Consider a symmetric tangentoid $D$ of $\E$. $U(M(D))$ coincides with $D$. Now, consider a tangentoid $(D,\mu^\prime,\eta^\prime)$ of $\cMon(\E)$. $M(U(D,+,0))$ comes equipped with the structure of a commutative monoid $(D,\mu,z)$. However, since $(D,\mu^\prime,\eta^\prime)$ is a tangentoid in $\cMon(\E)$ the structural morphisms $p,z,s,l$, and $c$ of $D$ must be monoid morphisms with respect to $\mu^\prime$ and $\eta^\prime$, so in particular, also $\mu$ and $z$ are monoid morphisms for $(D,\mu^\prime,\eta^\prime)$. Thus, using a Eckmann–Hilton argument, $(D,\mu,z)$ and $(D,\mu^\prime,\eta^\prime)$ must coincide as commutative monoids, proving that $(D,+,0)=M(U(D,\mu^\prime,\eta^\prime))$.
\end{proof}

\begin{remark}
\label{remark:U-preservation-limits}
In all cases of interest, the functor $U$ preserves limits. In particular, when $\E$ is cocomplete and $\x$ preserves colimits, $U$ is monadic. Moreover, when $\E$ has finite limits, $U$ creates limits; thus, in particular, it preserves limits.
\end{remark}

Lemma~\ref{lemma:tangentoids-are-monoids} shows that the monoid structure of a tangentoid $D$ on a symmetric monoidal category is always commutative, regardless of $D$ being symmetric. Therefore, one might hope every tangentoid in a symmetric monoidal category is symmetric. In particular, as shown above, this is true for coCartesian monoidal categories. However, despite our lack of a counterexample, we have not found any reason to believe that every tangentoid in a symmetric monoidal category must be symmetric. However, when the canonical flip is a monoid morphism, the tangentoid is symmetric.

\begin{lemma}
\label{lemma:c-monoid-morphism} Let $D$ be a tangentoid in a symmetric monoidal category $(\E,\x,\I, \sigma)$. Then $D$ is symmetric if and only if the canonical flip $c\colon D\x D\to D\x D$ is a monoid morphism. 
\end{lemma}
\begin{proof} For the $\Rightarrow$ direction, for any commutative monoid in a symmetric monoidal category, the symmetry isomorphism $\sigma$ is always a monoid morphism. For $\Leftarrow$ direction, let $D$ be a tangentoid whose canonical flip $c$ is a monoid morphism. In particular, this means that $c$ preserves the multiplication in the sense that the following diagram commutes:
\begin{equation*}
\begin{tikzcd}
{D\x D\x D\x D} & {D\x D\x D\x D} \\
{D\x D\x D\x D} & {D\x D\x D\x D} \\
{D\x D} & {D\x D}
\arrow["{c\x c}", from=1-1, to=1-2]
\arrow["{\id_D\x\sigma\x\id_D}"', from=1-1, to=2-1]
\arrow["{\id_D\x\sigma\x\id_D}", from=1-2, to=2-2]
\arrow["{\mu\x\mu}"', from=2-1, to=3-1]
\arrow["{\mu\x \mu}", from=2-2, to=3-2]
\arrow["c"', from=3-1, to=3-2]
\end{tikzcd}
\end{equation*}
However, precomposing each path of the diagram by $\id_D\x z\x z\x\id_D$, we obtain that the following diagrams commute:
\begin{equation*}
\begin{tikzcd}
{D\x D\x D\x D} & {D\x D\x D\x D} \\
{D\x D} & {D\x D\x D\x D} \\
{D\x D} & {D\x D}
\arrow["{c\x c}", from=1-1, to=1-2]
\arrow["{\id_D\x\sigma\x\id_D}", from=1-2, to=2-2]
\arrow["{\id_D\x z\x z\x\id_D}", from=2-1, to=1-1]
\arrow["{z\x\id_D\x\id_D\x z}"{description}, from=2-1, to=1-2]
\arrow["\sigma"', from=2-1, to=3-1]
\arrow["{\mu\x \mu}", from=2-2, to=3-2]
\arrow["{z\x\id_{D\x D}\x z}"{description}, from=3-1, to=2-2]
\arrow[equals, from=3-1, to=3-2]
\end{tikzcd}\quad\hfill
\begin{tikzcd}[column sep=huge]
{D\x D\x D\x D} & {D\x D} \\
{D\x D\x D\x D} & {D\x D} \\
{D\x D} & {D\x D}
\arrow["{\id_D\x\sigma\x\id_D}"', from=1-1, to=2-1]
\arrow["{\id_D\x z\x z\x\id_D}"', from=1-2, to=1-1]
\arrow[equals, from=1-2, to=2-2]
\arrow["{\mu\x\mu}"', from=2-1, to=3-1]
\arrow["{z\x\id_D\x\id_D\x z}"', from=2-2, to=2-1]
\arrow[equals, from=2-2, to=3-2]
\arrow["c"', from=3-1, to=3-2]
\end{tikzcd}
\end{equation*}
Therefore, $c=\sigma$, and so $D$ is symmetric as desired. 
\end{proof}


\section{Characterizing representable tangent structures for affine schemes}
\label{section:characterization}
The main objective of this paper is to classify representable tangent structures on the category of affine schemes $\Aff_R$. From our previous discussion, to achieve our goal, it suffices to classify infinitesimal objects in $\Aff_R$. Since $\Aff_R$ is equivalent to the opposite category $\cAlg_R$, thanks to Lemma~\ref{lemma:infinitesimal-objects-tangentoids}, infinitesimal objects in $\Aff_R$ coincide with coexponentiable tangentoids in $\cAlg_R$ (with respect to the coCartesian monoidal structure). 

\par Therefore, the first step towards our objective is to classify (non-necessarily coexponentiable) tangentoids in $\cAlg_R$. However, note that $\cAlg_R$ is precisely the category of commutative monoids for the category of $\Mod_R$, that is, $\cMon(\Mod_R)=\cAlg_R$. Thus, by Proposition~\ref{proposition:symmetric-tangentoids}, (symmetric) tangentoids of $\cAlg_R$ are equivalent to symmetric tangentoids of $\Mod_R$. Therefore, we reduce the problem to classifying symmetric tangentoids of $\Mod_R$.

\subsection{Symmetric tangentoids in \texorpdfstring{$R$-modules}{R-modules}}
\label{section:R-modules}
By spelling out the definition, one sees that a symmetric tangentoid in $\Mod_R$ consists of an $R$-module $D$ together with the following data:
\begin{description}
\item[Projection] An $R$-linear morphism $p\colon D\to R$;

\item[Zero morphism] An $R$-linear morphism $z\colon R\to D$;

\item[Sum morphism] An $R$-linear morphism $s\colon D\times_R D\to D$, where $-\times_R-$ represents the pullback of the projection $p$ along itself;

\item[Vertical lift] An $R$-linear morphism $l\colon D\to D\x_RD$;
\end{description}
such that the axioms (which do not involve the canonical flip) in Definition~\ref{definition:tangentoid} hold. In particular, we have that $(p\colon D\to R,z\colon R\to D,s\colon D\times_RD\to D)$ is an additive bundle in $\Mod_R$, that is, a commutative monoid in the slice category $\Mod_R/R$. However, every commutative monoid in $\Mod_R/R$ is of the form $R \oplus M$, for some $R$-module $M$, where the monoid structure corresponds precisely to the additive structure on $M$. Therefore, every additive bundle in $\Mod_R$ is in fact an Abelian group bundle, and thus every tangentoid is Rosick\'y. 

\begin{lemma}
\label{lemma:tangentoids-have-negatives}
If $D$ is a tangentoid of $\Mod_R$, then $D\cong R\+ \ker p$. Furthermore, $D$ is Rosick\'y.
\end{lemma}
\begin{proof}
Since $(p,z,s)$ is a commutative monoid in the slice category $\Mod_R/R$, using a similar argument of~\cite[Example~5]{beck:triples-algebras-cohomology}, it follows that $D$ is isomorphic to the $R$-module $R\+ \ker p$. Therefore, one can define $n \colon D \to D$ by conjugating $\id_R\+-\colon R\+ \ker p \to R\+ \ker p$, where $-$ sends $x\in \ker p$ to $-x$, which makes $D$ into a Rosick\'y tangentoid.
\end{proof}

Since we now know that every tangentoid must be of a certain form, the goal is now to characterize the structure on an $R$-module $M$ which makes $D=R\+M$ a symmetric tangentoid in $\Mod_R$. To this end, we may first extract more information on the module associated to a tangentoid. 

\par So, let $D$ be a symmetric tangentoid in $\Mod_R$ and let $M = \ker p$. We may call $M$ the $R$-module \textbf{associated to} the symmetric tangentoid $D$. For simplicity, let us assume that $D = R \+ M$. As such, it follows that the projection corresponds to the projection of the biproduct
\begin{align*}
&R\+M\xrightarrow{\pi_R}R
\end{align*}
the zero morphism corresponds to the injection of the biproduct
\begin{align*}
&R \xrightarrow{\iota_R} R \+ M
\end{align*}
and the sum of the morphism
\begin{align*}
&D_2=R\+(M\oplus M)\xrightarrow{\id_R\++}R\+M
\end{align*}
where $+\colon M\oplus M\to M$ denotes the sum of $M$. Therefore, most of the tangentoid structure is canonically established. Thus, it remains to explain what the vertical lift corresponds to on $M$. It turns out that this corresponds to a \textit{cosemigroup} structure on $M$. Recall that in a (symmetric) monoidal category, a \textbf{(commutative) semigroup} is a pair $(A, \mu)$ consisting of an object $A$ equipped with a morphism $\mu\colon A \otimes A \to A$ which is associative (and commutative). Dually, a \textbf{(cocommutative) cosemigroup} is a pair $(C, \delta)$ consisting of an object $C$ equipped with a morphism $\delta\colon C \to C \otimes C$ which is coassociative (and cocommutative). In $\Mod_R$, (co)commutative (co)semigroups correspond precisely to \textbf{non-(co)unital (co)commutative $R$-(co)algebras}.

In order to costruct a comultiplication on $M$, first recall that by Lemma~\ref{lemma:rosicky-tangentoids}, the following diagram is a triple equalizer in $\Mod_R$:
\begin{equation*}
\begin{tikzcd}
D & {D\x_RD} && D
\arrow["l", from=1-1, to=1-2]
\arrow["{\id_D\x_Rp}", bend left=30, from=1-2, to=1-4]
\arrow["{(z\o p)\x_Rp}"', bend right=30, from=1-2, to=1-4]
\arrow["{p\x_R\id_D}"{description}, from=1-2, to=1-4]
\end{tikzcd}
\end{equation*}
By expanding the tensor product $(R\+M)\otimes(R\+M)$, we find that:
\begin{align*}
&D\x_RD\cong R\+M\+M\+(M\x_RM)
\end{align*}
So, we may define an $R$-linear morphism:
\begin{align*}
&j\colon D(2)\=R\+(M\x_RM)\xrightarrow{}R\+M\+M\+(M\x_RM)=D\x_RD
\end{align*}
which sends $a\in R$ to $a$ and $x\x_Ry\in M\x_RM$ to $x\x_Ry$. From this morphism we can obtain a cocommutative cosemigroup on $M$, or in other words, make $M$ into a non-counital cocommutative $R$-coalgebra. 

\begin{lemma}
\label{lemma:lift}
Let $D$ be a symmetric tangentoid in $\Mod_R$, with $D = R \+ M$. Define the $R$-linear morphism $\nu\colon M\to M\x_RM$ as the following composite:
\begin{align*}
&\nu\colon M\xrightarrow{\iota_M}R\+M=D\xrightarrow{l}D\x_RD=(R\+M)\x_R(R\+ M)\xrightarrow{\pi_M\x_R\pi_M}M\x_RM
\end{align*}
Then $M$ is a cocommutative non-counital $R$-coalgebra with comultiplication $\nu\colon M \to M \otimes M$. Moreover, the following equality holds:
\begin{align*}
&l=j\o(\id_R\+\nu)
\end{align*}
\end{lemma}
\begin{proof}
For the sake of simplicity, for each $a,b\in R$ and $x,y\in M$ we shall denote by $ab$, $ax$, $by$, and $xy$ the elements $a\x_Rb$, $x\x_Ra$, $b\x_Ry$, and $x\x_Ry$ of $D\x_RD\cong R\+M\+M\+(M\x_RM)$, respectively. Using this notation, by linearity, we can write the action of the vertical lift $l$ as follows:
\begin{align*}
&l(a+x)=\sum_i(a_i+x_i)\x_R(b_i+y_i)=\sum_ia_ib_i+a_iy_i+b_ix_i+x_iy_i
\end{align*}
However, since $l$ triple equalizes $\id_D\x_Rp$, $p\x_R\id_D$, and $p\o z\x_Rp$, the terms $a_iy_i$ and $b_ix_i$ vanish. Therefore, we can rewrite $l$ as follows:
\begin{align*}
&l(a+x)=\sum_ia_ib_i+x_iy_i
\end{align*}
Thanks to the additivity axiom of the vertical lift, we get that
\begin{align*}
&l(a)=l(z(a))=(\id_D\x_Rz)(z(a))=a
\end{align*}
where we use that $a+0\in D$ for each $a\in R$. Thus, $l(a)=a$ for every $a\in R$. The next step is to prove that $l(x)$ belongs to $M\x_RM$ for every $x\in M$. From the zero axiom, we have that:
\begin{align*}
&(p\x_R\id_D)(l(x))=p(z(x))=0
\end{align*}
However, we also have that
\begin{align*}
&(p\x_R\id_D)(l(x))=(p\x_R\id_D)\left(B+\sum_ix_iy_i\right)=B
\end{align*}
for some $B\in R$ and $x_i,y_i\in M$. Thus, $B=0$, and so:
\begin{align*}
&l(x)=\sum_ix_iy_i
\end{align*}
Now we can also compute that
\begin{align*}
&j((\id_R\+\nu)(a+x))=j(a+\nu(x))=j\left(a+\sum_ix_iy_i\right)=a+\sum_ix_iy_i=l(a+x)
\end{align*}
where $\nu(x)$ is $\sum_ix_iy_i$. So $l=j\o(\id_R\+\nu)$ as desired. Finally, cocommutativity and coassociativity of $\nu$ are a direct consequence of the cocommutativity and coassociativity axioms of $l$.
\end{proof}

The universality of the vertical lift of $D$ implies the existence of a morphism $\alpha\colon M\x_RM\to M$, which defines an inverse of the comultiplication $\nu\colon M\to M\x_RM$. As such, this makes $M$ into a \textbf{solid} commutative non-unital $R$-algebra. Recall that in a monoidal category $(\E,\x,\I)$, a \textbf{solid} (commutative) semigroup is a (commutative) semigroup $(A,\mu)$ whose multiplication morphism $\mu\colon A \x A \to A$ is an isomorphism~\cite{gutierrez2015solid}.

\par Dually, a \textbf{solid} (cocommutative) cosemigroup is a (cocommutative) cosemigroup $(C, \delta)$ whose comultiplication morphism $\delta\colon C\to C\x C$ is an isomorphism. Of course, if $(A, \mu)$ is a solid (commutative) semigroup, then $(A, \mu^{-1})$ is a solid (cocommutative) cosemigroup, and vice-versa. Thus, solid (commutative) semigroups and solid (cocommutative) cosemigroups are essentially the same. In the context of the monoidal category $\Mod_R$ of $R$-modules, solid (commutative) semigroups correpond precisely to \textbf{solid} (\textbf{commutative}) \textbf{non-unital algebras}.

\par The full subcategory of the category of commutative semigroups on a symmetric monoidal category $(\E,\x,\I)$ spanned by commutative solid semigroups is denoted by $\cSolid(\E)$. In particular, a morphism between two commutative solid semigroups is just a morphism of semigroups. For more on solid semigroups, we invite the interested reader to see~\cite{gutierrez2015solid}. Here are some examples of commutative solid semigroups. 

\begin{example}
\label{example:solid-non-unital-algebra-unit}
In every symmetric monoidal category $(\E,\x,\I,\sigma)$, the unit $\I$ is trivally a commutative solid semigroup.
\end{example}

\begin{example}
\label{example:solid-non-unital-algebra-Q}
In the symmetric monoidal category $\Z$-modules, that is, Abelian groups, the ring of rational numbers $\Q$ is a commutative solid semigroup.
\end{example}

To prove that the $R$-module $M$ associated with a symmetric tangentoid $D$ of $\Mod_R$ is in fact a solid non-unital $R$-algebra, let us first define the multiplication $\alpha\colon M\x_RM\to M$ by employing the universal property of the lift $l$ of $D$. So consider the morphism $j$ and notice that $j$ triple equalizes $\id_D\x_Rp$, $p\x_R\id_D$, and $p\o z\x_Rp$. Therefore, by the universal property of $l$, we obtain a unique morphism $\hat j\colon D(2)\to D$ such that:
\begin{equation*}
\begin{tikzcd}
D & {D\x_RD} && D \\
{D(2)}
\arrow["l", from=1-1, to=1-2]
\arrow["{\id_D\x_Rp}", bend left=30, from=1-2, to=1-4]
\arrow["{(z\o p)\x_Rp}"', bend right=30, from=1-2, to=1-4]
\arrow["{p\x_R\id_D}"{description}, from=1-2, to=1-4]
\arrow["{\hat j}", dashed, from=2-1, to=1-1]
\arrow["{j}"', from=2-1, to=1-2]
\end{tikzcd}
\end{equation*}

 \begin{proposition}
\label{proposition:tangentoids-are-solid-commutative-non-unital-algebras}
Let $D$ be a symmetric tangentoid in $\Mod_R$, with $D = R \oplus M$. Define the $R$-linear morphism $\alpha\colon M\x_RM\to M$ as the following composite:
\begin{align*}
&\alpha\colon M\x_RM\xrightarrow{\iota_{M\x_RM}}R\+ (M\x_RM)=D(2)\xrightarrow{\hat j}D=R\+M\xrightarrow{\pi_M}M
\end{align*}
Then $M$ is a commutative solid non-unital $R$-algebra with multiplication $\alpha\colon M\x_RM\to M$ whose inverse is $\alpha^{-1} = \nu$ of Lemma~\ref{lemma:lift}. Moreover, the following equality holds:
\begin{align*}
&\hat j=\id_R\+\alpha
\end{align*}
\end{proposition}
\begin{proof}
As we already know that $M$ is a cocommutative non-counital coalgebra, it suffices to show that $\alpha^{-1} = \nu$, since commutativity and associativity of $\alpha$ come directly from the cocommutativity and coassociativity of $\nu$. To this end, we may first show that $\hat j=\id_R\+\alpha$. Observe that, or any $a \in R$, the following equality holds:
\begin{align*}
&l(\hat j(a))=j(a)=a = l(a)
\end{align*}
Since $l$ is monic, we get that $\hat j(a)=a$. Now, for any $x \otimes y \in M \otimes M$, let
\begin{align*}
&\hat j(x\x_Ry)=b+m
\end{align*}
for some $b\in R$ and $m\in M$. On the one hand we may compute:
\begin{align*}
&p(\hat j(x\x_Ry))=p(b+m)=b
\end{align*}
On the other hand, we may also compute
\begin{align*}
&p(\hat j(x\x_Ry))=((z\o p)\x_Rp)(j(x\x_Ry))=((z\o p)\x_Rp)(x\x_Ry)=0
\end{align*}
So, in particular, $b = 0$. Thus, $l^\circ(x\x_Ry)$ belongs to $M$, and so we get that:
\begin{align*}
&\hat j=\id_R\+\alpha
\end{align*}
as desired. Let us now show that $\nu\o\alpha=\id_{M\x_RM}$. By definition of $\hat j$:
\begin{align*}
&l\o \hat j=j
\end{align*}
However, $\hat j=\id_R\+\alpha$ and $l=j \o (\id_R\+\nu)$, thus:
\begin{align*}
&j \o (\id_R\+\nu) \o (\id_R\+\alpha) =j
\end{align*}
Since $j$ is monic, we get that:
\begin{align*}
&(\id_R\+\nu)\o(\id_R\+\alpha)=\id_{D(2)}
\end{align*}
So, in particular, $\nu\o\alpha=\id_{M\x_RM}$. Conversely, to show that $\alpha\o\nu = \id_M$, first consider the diagram:
\begin{equation*}
\begin{tikzcd}
D & {D\x_RD} && D \\
{D(2)} & D
\arrow["l", from=1-1, to=1-2]
\arrow["{\id_D\x_Rp}", bend left=30, from=1-2, to=1-4]
\arrow["{(z\o p)\x_Rp}"', bend right=30, from=1-2, to=1-4]
\arrow["{p\x_R\id_D}"{description}, from=1-2, to=1-4]
\arrow["{\id_R\+\alpha=\hat j}", from=2-1, to=1-1]
\arrow["{j}"{description}, from=2-1, to=1-2]
\arrow["l"', from=2-2, to=1-2]
\arrow["{\id_R\+\nu}", from=2-2, to=2-1]
\end{tikzcd}
\end{equation*}
From the universal property of $l$:
\begin{align*}
&(\id_R\+\alpha)\o(\id_R\+\nu)=\id_D
\end{align*}
which implies that $\alpha\o\nu=\id_M$. Therefore, we conclude that $M$ is a (co)commutative solid non-(co)unital (co)algebra as desired.
\end{proof}

We have now shown that if $D$ is a symmetric tangentoid in $\Mod_R$, the $R$-module $M$ associated with $D\cong R\+M$ comes equipped with the structure of a solid non-unital $R$-algebra. It is natural to wonder if the converse holds. For starters, we need a technical lemma.

\begin{lemma}
\label{lemma:preservation-pullbacks}
If $M$ is an $R$-module and $D\=R\+M$, the functors $D^n\x_R-$ preserve each $m$-fold pullback $D_m$ of the projection $p$ along itself, where $D^n$ denotes the $n$-fold tensor of $D$ with itself.
\end{lemma}
\begin{proof}
For an arbitrary $R$-module $N$, the diagram
\begin{equation*}
\begin{tikzcd}
{N\x_R(R\+M\+\dots\+M)} & {N\x_R(R\+M)} \\
{N\x_R(R\+M)} & N
\arrow["{N\x_R\pi_m}", from=1-1, to=1-2]
\arrow["{N\x_R\pi_1}"', from=1-1, to=2-1]
\arrow["{N\x_Rp}", from=1-2, to=2-2]
\arrow["{N\x_Rp}"', from=2-1, to=2-2]
\end{tikzcd}
\end{equation*}
is a pullback diagram once one distributes biproducts over the tensor products. Therefore, by choosing $N=D^n$ we get that $D^n$ preserves the $n$-fold pullbacks of $p$ along itself as desired.
\end{proof}

Thanks to Lemma~\ref{lemma:preservation-pullbacks}, we can prove the converse of Proposition~\ref{proposition:tangentoids-are-solid-commutative-non-unital-algebras}, that is, that commutative solid non-unital $R$-algebras define symmetric tangentoids in $\Mod_R$.

\begin{proposition}
\label{proposition:solid-non-unital-algebras-are-tangentoids}
If $M$ is a commutative solid non-unital $R$-algebra, $D\=R\+M$ is a symmetric tangentoid in $\Mod_R$.
\end{proposition}
\begin{proof}
The $R$-module structure of $M$ induces a projection $p\colon D\to R$, a zero morphism $z\colon R \to D$, a sum morphism $s\colon D_2\to D$, and a negation $n\colon D\to D$. Furthermore, thanks to Lemma~\ref{lemma:preservation-pullbacks}, the functors $D^n\x_R-$ preserve all $n$-fold pullbacks of $p$ along itself. The next step is to define the lift. Since $M$ is a solid non-unital $R$-algebra, $M$ comes equipped with an invertible morphism $\alpha\colon M\x_RM\to M$. Let $\nu$ denote the inverse of $\alpha$ and let us define:
\begin{align*}
&l\colon D\xrightarrow{\id_R\+\nu}D(2)\xrightarrow{j}D\x_RD
\end{align*}
where recall that $D(2) = R \+ (M \otimes M)$ and $j( r + (x \otimes y) ) = r + (x \otimes y)$. We now show that $l$ satisfies the necessary diagrams. Throughout this proof we represent:
\begin{equation*}
\nu(x)=\sum_ix_i\x_Ry_i\qquad
\nu(y)=\sum_jz_j\x_Rt_j
\end{equation*}
It is straightforward to see that $l$ satisfies the zero rule:
\begin{align*}
&(z\x_R\id_D)(l(a+x))=(z\x_R\id_D)(j((\id_\R\+\nu)(a+x)))=(z\x_R\id_D)\left(j\left(a+\left(\sum_ix_i\x_Ry_i\right)\right)\right)\\
&\quad=(z\x_R\id_D)\left(a+\sum_ix_i\x_Ry_i\right)=a=z(a+x)
\end{align*}
Since $l$ is a additive bundle morphism, we first compute that
\begin{align*}
& l(s(a+x_1+y_2))=l(a+(x+y))=l(a+x)+l(y)=a+\sum_ix_i\x_Ry_i+\sum_jz_j\x_Rt_j 
\end{align*}
while, we can also compute that:
\begin{align*}
&(\id_D\x_Rs)\left(\theta\left((l\times_Rl)\left(a+(x,0)+(0,y)\right)\right)\right)=(\id_D\x_Rs)\left(\theta\left(a+\sum_i(x_i,0)\x_R(y_i,0)+\sum_j(0,z_j)\x_R(0,t_j)\right)\right)\\
&\quad=(\id_D\x_Rs)\left(a+\left(\sum_ix_i\x_Ry_i,0\right)+\left(\sum_j0,z_j\x_Rt_j\right)\right)=a+\sum_ix_i\x_Ry_i+\sum_jz_j\x_Rt_j
\end{align*}
To conclude the proof, we need to show that $l$ is universal. Let us consider a morphism $f\colon A\to D\x_RD$ which triple equalizes $\id_D\x_Rp$, $p\x_R\id_D$, and $p\o z \x_Rp$. We want to find a morphism $h\colon A\to D$ making the following diagram commutative:
\begin{equation*}
\begin{tikzcd}
D & {D(2)} & {D\x_RD} && D \\
A
\arrow["{\id_R\+\nu}", from=1-1, to=1-2]
\arrow["{j}", from=1-2, to=1-3]
\arrow["{\id_D\x_Rp}", bend left, from=1-3, to=1-5]
\arrow["{p\o z\x_Rp}"', bend right, from=1-3, to=1-5]
\arrow["{p\x_R\id_D}", from=1-3, to=1-5]
\arrow["h", dashed, from=2-1, to=1-1]
\arrow["f"', from=2-1, to=1-3]
\end{tikzcd}
\end{equation*}
Since $f$ equalizes the three morphisms, $f(a)$ must be of the form:
\begin{align*}
&f(a)=A+\sum_ix_i\x_Ry_i
\end{align*}
Let us define the following morphism:
\begin{align*}
&l^\flat\colon D\x_RD\to D(2)\\
&l^\flat\left(\sum_i(a_i+x_i)\x_R(b_i+y_i)\right)\=\sum_ia_ib_i+x_i\x_Ry_i
\end{align*}
Then define $h$ as follows:
\begin{align*}
&h\colon C\xrightarrow{f}D\x_RD\xrightarrow{l^\flat}D(2)=R\+(M\x_RM)\xrightarrow{\id_R\+\alpha}D
\end{align*}
We now show that $l \o h=f$
\begin{align*}
&l(h(a))=j((\id_R\+\nu)((\id_R\+\alpha)(l^\flat(f(a)))))=j(l^\flat(f(a)))=j\left(l^\flat\left(A+\sum_ix_i\x_Ry_i\right)\right)\\
&\quad=j\left(A+\sum_ix_i\x_Ry_i\right)=A+\sum_ix_i\x_Ry_i=f(a)
\end{align*}
where we used that $\alpha$ is the inverse of $\nu$. For uniqueness of $h$, consider another morphism $h'\colon A\to D$ such that $l\o h'=f$. Then we can compute that:
\begin{align*}
&h=(\id_R\+\alpha)\o l^\flat\o f=(\id_R\+\alpha)\o\l^\flat\o l\o h'=(\id_R\+\alpha)\o l^\flat\o j\o(\id_R\+\nu)\o h'
\end{align*}
However, $l^\flat\o j=\id_{D(2)}$, and thus
\begin{align*}
&h=(\id_R\+\alpha)\o(\id_R\+\nu)\o h'=h'
\end{align*}
proving that $h$ is the unique morphism making the diagram commutative. So we conclude that $D = R \+ M$ is a symmetric tangentoid. 
\end{proof}

Propositions~\ref{proposition:tangentoids-are-solid-commutative-non-unital-algebras} and~\ref{proposition:solid-non-unital-algebras-are-tangentoids} provide an equivalence between symmetric tangentoids and commutative solid non-unital $R$-algebras.

\begin{theorem}
\label{theorem:characterization-tangentoids-Mod-R}
Symmetric tangentoids in $\Mod_R$ are (up to unique isomorphism) $R$-modules $R\+M$ where $M$ is a commutative solid non-unital $R$-algebra $M$ in $\Mod_R$. 
\end{theorem}

The correspondence of Theorem~\ref{theorem:characterization-tangentoids-Mod-R} extends to morphisms of symmetric tangentoids and commutative solid non-unital $R$-algebras. As a shorthand, let $\cSolid_R \= \cSolid(\Mod_R)$.

\begin{corollary}
\label{corollary:characterization-tangentoids-Mod-R}
The assignment defined by Proposition~\ref{proposition:tangentoids-are-solid-commutative-non-unital-algebras} which sends a symmetric tangentoid $D$ in $\Mod_R$ to a commutative solid non-unital $R$-algebra, $\Solid(D) = \ker p$ extends to an equivalence of categories:
\begin{align*}
&\Solid\colon\sTngoid(\Mod_R)\leftrightarrows\cSolid_R\colon R\+-
\end{align*}
\end{corollary}
\begin{proof}
Consider a morphism $f\colon D\to D'$ of symmetric tangentoids in $\Mod_R$. Since $D\cong R\+M$ and $D'\cong R\+M'$, and $p^\prime \o f = p$, it follows that $f$ is equivalent to a morphism $\id_R\+g$ for some $R$-module morphism $g\colon M\to M'$. Specifically, $g$ is the restriction of $f$ on $M\=\ker p$. Moreover, from the compatibility between $f$ and the vertical lifts, one can see that $g$ is a morphism of non-unital $R$-algebras. Therefore, $\Solid\colon\sTngoid(\Mod_R)\to\cSolid_R$ is a functorial assignment. The unit and the counit of the adjunction $\Solid\dashv U$ are respectively:
\begin{align*}
&\eta\colon D\cong R\+ M && \epsilon\colon M\cong\ker(\pi_M\colon R\+ M\to M)=\Solid(R\+ M)
\end{align*}
From here, it is straightforward to see that this gives an equivalence of categories as desired. 
\end{proof}

\subsection{Tangentoids for algebras}
\label{section:tangentoids-algebras}
Theorem~\ref{theorem:characterization-tangentoids-Mod-R} characterizes symmetric tangentoids in $\Mod_R$. Thanks to Proposition~\ref{proposition:symmetric-tangentoids}, symmetric tangentoids in $\Mod_R$ are equivalent to tangentoids in $\cMon(\Mod_R)=\cAlg_R$. In this section, we use this correspondence to characterize tangentoids in $\cAlg_R$. Moreover, it turns out that the canonical monoid structure on a tangentoid corresponds to the algebra structure of a semidirect product. Recall that for an $R$-module $M$, the \emph{semidirect product}, also known as square-zero extension, is the commutative $R$-algebra $R \ltimes M$ whose underlying $R$-module is $R \ltimes N \= R \+ N$ equipped with multiplication $(a+m)(b+n) = ab + am + bn$, for all $a,b \in R$ and $m, n \in M$. 

\begin{theorem}
\label{theorem:characterization-tangentoids}
If $D$ is a tangentoid in $\cAlg_R$, then $D\cong R\ltimes M$, where $M \= \ker p$, and $M$ comes equipped with the structure of a commutative solid non-unital $R$-algebra $\alpha\colon M\x_RM\to M$. Conversely, each commutative solid non-unital $R$-algebra $M$ defines a tangentoid $D\=R\ltimes M$ in $\cAlg_R$.
\end{theorem}
\begin{proof}
By Proposition~\ref{proposition:symmetric-tangentoids}, a tangentoid $D$ in $\cAlg_R$ is equivalent to a symmetric tangentoid of $\Mod_R$ equipped with the commutative monoid structure of Lemma~\ref{lemma:tangentoids-are-monoids}. By unwrapping this monoid structure, it is easy to show that a tangentoid $D$ of $\cAlg_R$ coincides with the semidirect product $R\ltimes M$ between the ring $R$ and a commutative solid non-unital $R$-algebra $M$.
\end{proof}

\begin{corollary}
\label{corollary:characterization-tangentoids}
The assignment defined by Theorem~\ref{theorem:characterization-tangentoids} which sends a tangentoid of $\cAlg_R$ to a commutative solid non-unital $R$-algebra extends to an equivalence of categories:
\begin{align*}
&\Solid\colon\Tngoid(\cAlg_R)\leftrightarrows\cSolid_R\colon R\ltimes-
\end{align*}
\end{corollary}

Thanks to the characterization of tangentoids in $\cAlg_R$, we can now construct new examples of tangent structures on $\cAlg_R$. In particular, when $R = \Z$, $\cAlg_\Z$ is precisely $\cRing$, the category of commutative rings. 

\begin{example}
\label{example:Q-tangentoid-non-exponentiable}
As discussed in Example~\ref{example:solid-non-unital-algebra-Q}, $\Q$ comes equipped with a commutative solid semigroup structure in the category of Abelian groups. Therefore, by Theorem~\ref{theorem:characterization-tangentoids}, $\Z\ltimes\Q$ defines a tangentoid in $\cRing$. This induces a tangent structure $[(\Z\ltimes\Q) \x_R-]$ on $\cRing$, whose tangent bundle functor can equivalently be described by $T(R) = R \ltimes (\Q \otimes R)$. 
\end{example}

\subsection{Infinitesimal objects for affine schemes}
\label{section:infinitesimal-objects}
Thanks to Lemma~\ref{lemma:infinitesimal-objects-tangentoids}, infinitesimal objects, and therefore representable tangent structures, on $\Aff_R$ correspond to coexponentiable tangentoids in $\cAlg_R$. To better characterize these tangentoids, let us first recall Niefield's theorem that classifies the coexponentiable objects in $\cAlg_R$. 

\begin{theorem}
\label{lemma:nielfield-characterization-exponential-affine-scheme}
\cite[Theorem~4.3]{niefield:cartesianness} A commutative $R$-algebra $A$ is coexponentiable in $\cAlg_R$ if and only if its underlying $R$-module is finitely generated and projective.
\end{theorem}

Thanks to Theorem~\ref{lemma:nielfield-characterization-exponential-affine-scheme}, we can finally characterize coexponentiable tangentoids of $\cAlg_R$. By a \textbf{finitely generated projective} (\textbf{commutative}) \textbf{solid non-unital $R$-algebra}, we mean a solid (commutative) non-unital $R$-algebra whose underlying $R$-module is finitely generated and projective. 

\begin{theorem}
\label{theorem:characterization-infinitesimal-objects}
The coexponentiable tangentoids of $\cAlg_R$ are precisely the $R$-algebras $D=R\ltimes M$ associated with a finitely generated projective commutative solid non-unital $R$-algebra $M$.
\end{theorem}
\begin{proof}
Theorem~\ref{theorem:characterization-tangentoids} characterizes tangentoids of $\cAlg_R$ in terms of $R$-algebras $D\cong R\ltimes M$ where $M$ is a solid commutative non-unital $R$-algebra. By definition, a tangentoid $D=R\ltimes M$ is coexponentiable provided that each $D_n$ is coexponentiable. However, since $\cAlg_R$ is complete, by Lemma~\ref{lemma:Frankland-lemma-representability}, a tangentoid $D$ is coexponentiable in $\cAlg_R$ if and only if the underlying object $D$ is coexponentiable. Moreover, Theorem~\ref{lemma:nielfield-characterization-exponential-affine-scheme} establishes that coexponentiable $R$-algebras are equivalent to finitely generated and projective $R$-modules. However, $R\+M$ is a finitely generated and projective $R$-modules if and only if so is $M$. To show this, notice that $M$ is projective if and only if there exists an $R$-module $N$ such that $M\+N$ is a free $R$-module. However, if $M$ is projective, then $R\+M\+N$ is free, so $R\+M$ is projective. Conversely, if $R\+M$ is projective, then $R\+M\+N$ is free, thus, there exists a module $R\+N$ for which $M\+(R\+N)$ is free; thus, $M$ is projective. Finally, it is immediate to see that $M$ is finitely generated if and only if so is $R\+M$. Thus, coexponentiable tangentoids are classified by finitely generated projective commutative solid non-unital $R$-algebras. 
\end{proof}

Since coexponentiable tangentoids correspond to infinitesimal objects in the opposite category, we also have a characterization of representable tangent structures on $\Aff_R$.

\begin{corollary}
\label{corollary:classication-representable-tangent-structures}
A tangent structure $\TT$ on $\Aff_R$ is representable if and only if there is a finitely generated projective commutative solid non-unital $R$-algebra $M$ such that the tangent bundle functor $\T$ is right adjoint to the functor $D\x_R-$, where $D\=R\ltimes M$.
\end{corollary}

When the base ring $R$ is a principal ideal domain (PID), the characterization of Theorem~\ref{theorem:characterization-infinitesimal-objects} drastically simplifies, thanks to the following well-known fact. 

\begin{lemma}
\label{lemma:PID-projective-free-modules}~\cite[Theorem~9.8]{rotman:advanced-modern-algebra} The projective $R$-modules of a PID ring $R$ are free.
\end{lemma}

By applying Lemma~\ref{lemma:PID-projective-free-modules} we find that there are only two coexponentiable tangentoids in $\cAlg_R$ when $R$ is a PID.

\begin{theorem}
\label{theorem:PID-characterization}
Let $R$ be a PID. Then the only coexponentiable tangentoids in $\cAlg_R$ is $R$ itself and the ring of dual numbers $R[\epsilon]$.
\end{theorem}
\begin{proof}
By Theorem~\ref{theorem:characterization-infinitesimal-objects}, the $R$-module $M$ associated with a coexponentiable tangentoid of $\cAlg_R$ is a finitely generated projective commutative solid non-unital $R$-algebra. Since $M$ is projective, by Lemma~\ref{lemma:PID-projective-free-modules}, $M$ is free. Moreover, since $M$ is finitely generated, $M$ is isomorphic to $n$ copies of $R$, for some non-negative integer $n$, that is, $M\cong R^n$. Since $M$ is a solid non-unital $R$-algebra we have the following chain of isomorphisms:
\begin{align*}
&R^{n^2}\cong R^n\x_RR^n\cong M\x_RM\cong M\cong R^n
\end{align*}
which holds only if $n^2=n$, that is when $n=0,1$. Therefore, $M=0$ or $M=R$, meaning that $D=R$ or $D=R\ltimes R=R[\epsilon]$. 
\end{proof}

Using again that coexponentiable tangentoids are infinitesimal objects in the opposite category, there are only two representable tangent structures on $\Aff_R$ when $R$ is a PID.

\begin{corollary}
\label{corollary:PID-infinitesimal-are-R-D}
When the base ring $R$ is a PID, there are precisely two representable tangent structures on $\Aff_R$, corresponding to the trivial tangent structure (Example~\ref{example:trivial}) and the tangent structure $\TT^\Omega$ of Example~\ref{example:affine-canonical}.
\end{corollary}
\begin{proof}
Thanks to Theorem~\ref{theorem:PID-characterization}, the only two coexponentiable tangentoids of $\cAlg_R$ are $R$ and $R[\epsilon]$, which correspond precisely to the trivial tangent structure on $\Aff_R$ and $\TT^\Omega$, respectively.
\end{proof}

In particular, when $R$ is the ring of integers $\Z$, $R$ is a PID, therefore we have the following result.

\begin{corollary}
\label{corollary:cRing-only-R-D}
In the category of affine schemes $\Aff$, the only representable tangent structures are the trivial one and $\TT^\Omega$.
\end{corollary}

On the one hand, it is important to stress that there could be other non-coexponentiable tangentoids in $\cAlg_R$, even when the base ring is a PID. 

\begin{example}
\label{example:Q-tangentoid-non-exponentiable-2} While $\Z\ltimes\Q$ is a tangentoid in $\cRing$, since $\Q$ is not finitely generated as a $\Z$-module, this tangentoid is not coexponentiable. Therefore, the induced tangent structure is not adjunctable, meaning that it does not induce a representable tangent structure on the category $\Aff$.
\end{example}

On the other hand, it is natural to wonder if there are any other coexponentiable tangentoid when the base ring fails to be a PID.

\begin{example}
\label{example:non-PID-example}
Let $S$ be any commutative ring. The product ring $R\=S\times S$ is not a PID since it has a zero divisor, e.g., $(1,0)$. It is straightforward to see that $S$ is an $R$-module with action:
\begin{align*}
&(a,b)\.x\=ax
\end{align*}
for $(a,b)\in R$ and $x\in S$. Furthermore, $S\x_RS$ is the $R$-module freely generated by terms of type $x\x_Ry$, satisfying the following relation:
\begin{align*}
&((a,b)\.x)\x_Ry=x\x_R((a,b)\.y)
\end{align*}
Therefore:
\begin{align*}
&x\x_Ry=((x,0)\.1)\x_Ry=1\x_R((x,0)\.y)=\1\x_R(xy)
\end{align*}
So, $S\x_RS\cong S$, and it is easy to check that this coincides with the multiplication map of $S$. Thus, $S$ is a commutative solid non-unital $R$-algebra. Furthermore, since $S\times S=R$ is free, $S$ is projective and finally, $S$ is finitely generated, since:
\begin{align*}
&S=R/(1,0)
\end{align*}
Therefore, $S$ induces a coexponentiable tangentoid $R\ltimes S$ in $\cAlg_R$, which is neither $R$ or $R[\epsilon]$.
\end{example}



\end{document}